\documentclass{amsart}

\usepackage{amsmath,amssymb,amsthm,graphicx}

\newtheorem{theorem}{Theorem}[section]
\newtheorem{lemma}[theorem]{Lemma}
\newtheorem{prop}[theorem]{Proposition}

\theoremstyle{definition}
\newtheorem{definition}[theorem]{Definition}
\newtheorem{example}[theorem]{Example}

\theoremstyle{remark}

\numberwithin{equation}{section}

\begin{document}

\title{General construction of symmetric parabolic structures}

\author{Jan Gregorovi\v c}



\thanks{The author would like to mention discussions with B. Doubrov, J. Slovak and L. Zalabova. This research has been supported by the grant GACR 201/09/H012.}

\dedicatory{}

\begin{abstract}
First we introduce a generalization of symmetric spaces to parabolic geometries. We provide construction of such parabolic geometries starting with classical symmetric spaces and we show that all regular parabolic geometries with smooth systems of involutive symmetries can be obtained this way. Further, we investigate the case of parabolic contact geometries in great detail and we provide the full classification of those with semisimple groups of symmetries without complex factors. Finally, we explicitly construct all non-trivial contact geometries with non-complex simple groups of symmetries. We also indicate geometric interpretations of some of them.
\end{abstract}

\maketitle

\section{Introduction}

In this section we recapitulate basic facts about Cartan connections and
symmetric spaces.  We show, that there are various Cartan geometries over
symmetric spaces and define symmetric parabolic geometries.  In the second
section we introduce a general construction of parabolic contact geometries
with smooth system of symmetries and show (Theorem \ref{them1}) that under
certain conditions we can construct all of them.  In the third section we
will deal
in detail with the construction for parabolic contact geometries
and show how to classify them.  Since the three dimensional case is
specific, we treat it separately in section four. The remaining sections deal with
parabolic contact geometries in general dimensions.

\subsection{Cartan connections}

Let $L$ be a Lie group and $P$ a closed Lie subgroup of $L$. There is a $P$-principal bundle $L\to L/P$ with the Maurer-Cartan form, which is a $\mathfrak{l}$-valued $1$-form identifying $\mathfrak{l}=T_eL$ with the left invariant vector fields on $L$. The Cartan geometry is generalization of this concept, for details and proofs look in \cite[Chapter 5]{odk5} and \cite[Chapter 1.5]{odk6}.

\begin{definition}
A Cartan geometry of type $(L,P)$ is a $P$-principal fiber bundle $p: \mathcal{P}\to M$ with $\mathfrak{l}$-valued $1$-form $\omega$ satisfying:

1) $\omega$ is $P$-equivariant, i.e. $(r^h)^*\omega=Ad(h^{-1})\circ \omega$ for $h\in P$

2) $\omega$ reproduces generators of the fundamental vector fields of the $P$ action

3) $\omega$ defines an absolute parallelism, i.e. $\omega|_{T_u\mathcal{P}}$ is a linear isomorphism.

A morphism of Cartan connections of type $(L,P)$ is a principal bundle morphism $\Phi: \mathcal{P}\to \mathcal{P}'$ such, that $(\Phi)^*\omega'=\omega$.

We say that two Cartan connections of type $(L,P)$ on $\mathcal{P}$ are equivalent if there is a principal bundle morphism $\Phi: \mathcal{P}\to \mathcal{P}$ such, that $(\Phi)^*\omega'=\omega$.
\end{definition}

The homogeneous space $L\to L/P$ is called the homogeneous model of Cartan
geometry of type $(L,P)$.

The constant vector fields 
$\omega^{-1}(X)$ for $X\in \mathfrak{l}$ play the role of 
the left invariant vector fields on the homogeneous model. 
We say that a Cartan
geometry is complete if the constant vector fields are complete.  The
difference from the homogeneous model is measured by the curvature form 
\[ K(\mu,\nu)=d\omega(\mu,\nu)+[\omega(\mu),\omega(\nu)].\]
Equivalently, evaluating on the constant vector fields we obtain 
the curvature function 
\[ \kappa(u)(X,Y)=K(\omega^{-1}(X)(u),\omega^{-1}(Y)(u))=[X,Y]-\omega([\omega^{-1}(X),\omega^{-1}(Y)])(u)\]
i.e.~it encodes the difference between the Lie bracket of $\mathfrak{l}$ and the
bracket of the constant vector fields on $\mathcal{P}$. Thus, $\kappa$ can be
viewed as a function 
\[ \kappa: \mathcal{P}\to \bigwedge^2(\mathfrak{l}/\mathfrak{p})^*\otimes \mathfrak{l}. \]

The morphisms of Cartan geometry always cover local 
diffeomorphisms $M\to M'$. The following important proposition is called the
Liouville theorem in the literature:

\begin{prop}\label{authomod}\cite[1.5.2]{odk6}
If $L/P$ is connected, then all locally defined automorphisms of the 
homogeneous model $L\to L/P$
are restrictions of the left multiplications by elements of $L$.
\end{prop}

Let us define the following subcategory of Cartan connections:

\begin{definition}
A Cartan geometry $p: \mathcal{P}\to M$ of type $(L,P)$ is called
homogeneous if there is a subgroup $G$ of the 
Lie group of automorphisms of the Cartan
geometry, that acts transitively on $M$, i.e.~$M=G/K$ is homogeneous space
for the isotropic subgroup $K\subset G$ of a point in $M$.
\end{definition}

\subsection{Parabolic geometries}
The general theory of parabolic geometries can be found in the detailed
exposition in \cite{odk6}. We shall briefly remind some of its features.

A parabolic geometry is a Cartan geometry $(p: \mathcal{P}\to M,\omega)$ of
type $(L,P)$, where $P$ is a parabolic subgroup of the semisimple group $L$.  
The parabolic subgroup enjoys a decomposition into its reductive part $L_0$
and the nilpotent part. The explicit choice of $L_0$ in its conjugacy class
provides the Lie algebra $\mathfrak{l}$ with grading 
\[ \mathfrak{l}=\oplus^k_{i=-k}\mathfrak{l}_i,\]
where 
$ \mathfrak{p}=\oplus^k_{i=0}\mathfrak{l}_i $ is the 
non-negative part of the gradation.  

This gradation defines
a filtration on the principal fibre bundle $\mathcal P$ via $\omega$, 
which descends to a
filtration $T^iM$ on $M$, possibly with further reduction of the structure
group of the graded tangent bundle to the group $L_0$.  

The parabolic geometries can be reconstructed from these underlying
structures, under suitable normalization conditions on the curvature. 
The basic assumption on the curvature is the regularity:

\begin{definition}
A parabolic geometry $p: \mathcal{P}\to M$ of type $(L,P)$ is called regular
if $\kappa(\mathfrak{l}_i,\mathfrak{l}_j)\in \oplus^k_{l=i+j+1}\mathfrak{l}_l$
for all $i,j< 0$.
\end{definition}

On the manifold $M$ itself this means, that the Lie brackets of vector
fields turn $TM$ into a filtered manifold and the associated algebraic Lie
bracket on the associated graded tangent spaces coincides with the bracket
inherited from the Lie algebra $\mathfrak l_{-k}\oplus\dots\oplus \mathfrak
l_{-1}$. 

We will need the following property of the graded Lie algebras in question
\cite[3.1.2.(4)]{odk6}:

\begin{lemma}\label{lemreg1}
Let $\mathfrak{l}_i$ be grading of a semisimple Lie algebra $\mathfrak{l}$. Then for $i<0$ we have $[\mathfrak{l}_{i+1},\mathfrak{l}_{-1}]=\mathfrak{l}_{i}$.
\end{lemma}

In order to reconstruct the Cartan geometries from such underlying
structures, we need further normalization on the curvature, which comes from
cohomological considerations at the level of the Lie algebra $\mathfrak l$.
The resulting Cartan connections are called normal and we shall discuss
the normality conditions only in the special cases later on. But the crucial
point is the
fact that the entire curvature of normal geometries is fully determined  
by smaller parts called harmonic
curvature $\kappa_H$.

As an example we mention the contact two gradings and the parabolic contact
structures.

\begin{example}
A contact $2$-grading is grading
$\mathfrak{l}=\mathfrak{l}_{-2}+\mathfrak{l}_{-1}+\mathfrak{l}_{0}+\mathfrak{l}_{1}+\mathfrak{l}_{2}$
such, that $ \operatorname{dim}(\mathfrak{l}_{\pm 2})=1$
and the Lie bracket 
$ \mathfrak{l}_{-1}\times \mathfrak{l}_{-1}\to \mathfrak{l}_{-2}$
is non-degenerate.  

According to \cite[Chapter 3.2.2]{odk6} the contact
two gradings can appear only on the simple Lie algebras, and the full list
of them follows. Here the
representation means the adjoint representation of the semisimple part of
$\mathfrak{l}_0$ on $\mathfrak{l}_{-1}$.  The parabolic contact structures
of dimension $d=\frac12(\operatorname{dim}\mathfrak{l}-\operatorname{dim}
\mathfrak{l}_0)$ are the
parabolic geometries corresponding to these gradings i.e.  they are uniquely
given by the $\mathfrak{l}$. The following table lists all contact parabolic
geometries and summarizes their properties for dimensions $d>3$.

\medskip
\begin{tabular}{|c|c|c|c|}
\hline
$\mathfrak{l}$& $\mathfrak{l}_0$ & representation & $\kappa_H$\\
\hline
$\mathfrak{sl}(n+2,\mathbb{R})$ & $\mathfrak{sl}(n,\mathbb{R})+\mathbb{R}^2$ & $\lambda_1 \oplus \lambda_{n-1}$ & $t$, $c$ \\
\hline
$\mathfrak{su}(p+1,q+1)$ &  $\mathfrak{su}(p,q)+\mathbb{R}^2$ & $\lambda_1 \oplus \lambda_{n-1}$ & $t$, $c$  \\
\hline
$\mathfrak{so}(p+2,q+2)$ &  $\mathfrak{so}(p,q)+\mathfrak{sl}(2,\mathbb{R})+\mathbb{R}$ & $\lambda_1\otimes \lambda_1(\mathfrak{sl})$ & $t$  \\
\hline
$\mathfrak{sp}(2n+2,\mathbb{R})$ &  $\mathfrak{sp}(2n,\mathbb{R})+\mathbb{R}$ & $\lambda_1$ & $c$  \\
\hline
$\mathfrak{so}^\star(2n+2)$ &  $\mathfrak{so}^\star(2n)+\mathfrak{su}(2)+\mathbb{R}$ & $\lambda_1\otimes \lambda_1(\mathfrak{su})$ & $t$  \\
\hline
$\mathfrak{g}_2(2)$ &  $\mathfrak{sl}(2,\mathbb{R})+\mathbb{R}$ & $3\lambda_1$ & $t$ \\
\hline
$\mathfrak{f}_4(4)$ &  $\mathfrak{sp}(6,\mathbb{R})+\mathbb{R}$ & $\lambda_3$ & $t$ \\
\hline
$\mathfrak{e}_6(6)$ &  $\mathfrak{sl}(6,\mathbb{R})+\mathbb{R}$ & $\lambda_3$ & $t$  \\
\hline
$\mathfrak{e}_6(2)$ &  $\mathfrak{su}(3,3)+\mathbb{R}$ & $\lambda_3$ & $t$  \\
\hline
$\mathfrak{e}_6(-14)$ &  $\mathfrak{su}(1,5)+\mathbb{R}$ & $\lambda_3$ & $t$  \\
\hline
$\mathfrak{e}_7(7)$ &  $\mathfrak{so}(6,6)+\mathbb{R}$ & $\lambda_6$ & $t$  \\
\hline
$\mathfrak{e}_7(-5)$ &  $\mathfrak{so}^\star(12)+\mathbb{R}$ & $\lambda_6$ & $t$  \\
\hline
$\mathfrak{e}_7(-25)$ &  $\mathfrak{so}(2,10)+\mathbb{R}$ & $\lambda_6$ & $t$  \\
\hline
$\mathfrak{e}_8(8)$ &  $\mathfrak{e}_7(7)+\mathbb{R}$ & $\lambda_7$ & $t$  \\
\hline
$\mathfrak{e}_8(-24)$ &  $\mathfrak{e}_7(-25)+\mathbb{R}$ & $\lambda_7$ & $t$  \\
\hline
\end{tabular}

\medskip
The third column records the representation in terms of the fundamental
representations of the semisimple part of $\mathfrak l_0$.
The harmonic components of the curvature $\kappa_H$ are indicated in the
last column. They consist of two types, 
torsion $(\bigwedge^2
\mathfrak{l}_{-1}^*)\otimes \mathfrak{l}_{-1}$ and curvature $(\bigwedge^2
\mathfrak{l}_{-1}^*)\otimes \mathfrak{l}_{0}$.  They are denoted $t$ and $c$
in the table, if the geometry has the harmonic curvature of that type.  If
the $t$ vanishes the geometry is torsion free and if both vanish, then the
curvature $\kappa$ vanishes.
\end{example}

Generalization of the symmetric spaces to the parabolic
geometries was recently introduced by Zalabov\'a 
in \cite{odk1} and \cite{odk13}:

\begin{definition}\label{symLenka}
We say that a parabolic geometry $(\mathcal{P}\to M,\omega)$ of type $(L,P)$
is symmetric, if for any point $x$ there is a diffeomorphism $S_x: M\to M$
such, that

1) $S_xx=x$

2) $T_xS_x|_{T^{-1}_xM}=-id_{T^{-1}_xM}$

3) $S_x$ is covered by an automorphism of Cartan geometry.
\end{definition}

She also noticed that the existence of symmetries in all points 
pose extremely strong
conditions on the curvature. In particular, the torsion has to vanish
completely in the contact parabolic case. 

In order to construct symmetric contact parabolic geometries below, we shall
heavily exploit the following functorial constructions. 

The first construction is called extension.

\begin{theorem}\cite[1.5.15]{odk6}
Let $(\mathcal{G}\to M,\omega)$ be a Cartan geometry of type $(G,K)$ and let $P$ be a closed subgroup of a Lie group $L$ such, that $dim(L/P)=dim(G/K)$. Let $i: K\to P$ be a Lie group homomorphism and $\alpha:\mathfrak{g}\to \mathfrak{l}$ a linear map such, that 

(i) $\alpha$ is $Ad(K)$-equivariant, i.e. $\alpha$ is isomorphism of Lie algebra representations

(ii) $\alpha|_{\mathfrak{k}}=i'$ 

(iii) $\alpha:\mathfrak{g}/\mathfrak{k}\to \mathfrak{l}/\mathfrak{p}$ is a vector space isomorphism.

Then there is a Cartan geometry $(\mathcal{G}\times_i P,\alpha \circ \omega)$ of type $(L,K)$ with curvature 
\[ \kappa_\alpha = \alpha \circ \kappa+[\alpha,\alpha]-\alpha([\ ,\ ]).\]
 This construction is functorial and conjugation by elements of $P$ defines a natural transformation of corresponding functors.
\end{theorem}

We will use the following proposition to determine, how many conjugacy classes of homomorphisms $i: K\to P$ can exist.

\begin{prop}\label{pp30}
Let $P$ be one of $Sl(n,\mathbb{R})$, $SU(p,q)$ or $Sp(2n,\mathbb{R})$ and let $K$ be a semisimple Lie group. Let $i,\ j: K\to P$ be two homomorphisms of Lie groups with discrete kernels such, that restrictions of standard representations $\mathbb{R}^n$ to $i(K)$ and $j(K)$ are isomorphic and irreducible. Then there is $C\in P$ such, that $i(k)=Cj(k)C^{-1}$ for all $k\in K$.
\end{prop}
\begin{proof}
We will use the general concept described in \cite{odk16}. After complexification to $P_{\mathbb{C}},\ K_{\mathbb{C}}$, we are in situation of \cite{odk16}[Chapter 6, proposition 3.2]. Thus there is $C\in P_{\mathbb{C}}$ such, that $i(k)=Cj(k)C^{-1}$ for all $k\in K_{\mathbb{C}}$. Let $\theta$ be the involutive automorphism fixing the real form $P$, then $Cj(k)C^{-1}=i(k)=\theta(i(k))=\theta(Cj(k)C^{-1})=\theta(C)j(k)\theta(C)^{-1}$ for all $k\in K$. Thus $C^{-1}\theta(C)$ commutes with all elements in $j(K)$. Since $i(K)$ acts irreducibly on $\mathbb{R}^n$, $C^{-1}\theta(C)$ has to act as multiple of identity by Schur's lemma, thus $\theta(C)C^{-1}=e$ and $\theta(C)=C$ i.e. $C\in P$.
\end{proof}

The second construction is called correspondence space construction.

\begin{theorem}\cite[1.5.13]{odk6}
Let $(\mathcal{G}\to M,\omega)$ be a Cartan geometry of type $(G,H)$ and let $K$ be a closed subgroup of a Lie group $H$. Then there is a Cartan geometry $(\mathcal{G}\to \mathcal{G}/K,\omega)$ of type $(G,K)$ with the same curvature. This construction is functorial and if $H/K$ is connected, then it is equivalence onto subcategory.
\end{theorem}

If we begin with homogeneous model, then the extension creates a homogeneous Cartan geometry. In fact we obtain all homogeneous Cartan geometries this way.

\begin{prop}\label{them2}\cite[1.1]{odk10}
Let $p: \mathcal{P}\to G/K$ be a homogeneous Cartan geometry of type $(L,P)$. Then it is extension of homogeneous model $G\to G/K$.
\end{prop}

Let $(\mathcal{P}\to M,\omega)$ be a Cartan geometry of type $(L,P)$. The
adjoint tractor bundle is the associated vector bundle
$\mathcal{A}M=\mathcal{P}\times_P \mathfrak{l}$ for the adjoint action of
$P$ on $\mathfrak{l}$ and there is a natural projection $\Pi:
\mathcal{A}M\to TM$.  Further there is an isomorphism between smooth
sections of $\mathcal{A}M$ and $P$-invariant (for the principal right 
action of $P$)
vector fields on $\mathcal{P}$.  The curvature 
$\kappa$ can also be viewed as an
$\mathcal{A}M$-valued two form on $M$. 
Then $T= \Pi \circ \kappa$ is called the
torsion of the Cartan geometry.

On homogeneous model, the flows of right invariant vector fields define one
parameter subgroups of automorphisms.  In the general case only some right
invariant vector fields define one parameter subgroups of automorphisms.

\begin{definition}
Infinitesimal automorphism of a Cartan geometry 
is a smooth section of $\mathcal{A}M$ such, that the flow of the corresponding $P$-invariant vector field is an one parameter subgroup of automorphisms.
\end{definition}

The following theorem gives us maximal estimate for automorphism groups.

\begin{theorem}\cite[1.5.11]{odk6}
The group of automorphisms of Cartan geometry is a Lie group with Lie algebra consisting of complete infinitesimal automorphisms and any infinitesimal automorphism is determined by its value at single point.
\end{theorem}

There is the fundamental derivative $D_s$ for any section $s$ of
$\mathcal{A}M$ on any natural associated vector bundle to $\mathcal{P}$. 
The $D_s$ is given by the derivation in the direction of the $P$-invariant
vector field corresponding to $s$.  Using this, we can define a linear
connection on $\mathcal{A}M$ as 
\[ \nabla^{inf}_{\Pi(s_1)}s_2=D_{s_1}s_2+\{s_1,s_2\}-\kappa(\Pi(s_1),\Pi(s_2)),\]
where $\{s_1,s_2\}(u)=[s_1(u),s_2(u)]$.  It can be shown, that infinitesimal
automorphisms are parallel with respect to $\nabla^{inf}$. Then the bracket of  infinitesimal automorphisms is:
\[[s_1,s_2]=\kappa(\Pi(s_1),\Pi(s_2))-\{s_1,s_2\}\]

For the homogeneous Cartan geometry we get the following:

\begin{prop}\label{lab_1}
Let $p: \mathcal{P}\to G/K$ be a complete homogeneous Cartan geometry of type $(L,P)$ given by extension $(i,\alpha)$ from the homogeneous model $G\to G/K$, where $G/K$ is simply connected. Then $\nabla^{inf}$ is the induced connection by 
\[ [\alpha^*,.\,]=[\alpha,.\,]+\kappa(.\,,\alpha): \mathfrak{g}\to \mathfrak{gl}(\mathfrak{l}).\]

The holonomy algebra of $\nabla^{inf}$ is 

\[ hol(\nabla^{inf})=R+[\alpha^*,R]+[\alpha^*,[\alpha^*,R]]+\dots ,\]

where $R$ is curvature of $\nabla^{inf}$. The following holds:

1) all infinitesimal automorphisms are given by $X\in \mathfrak{l}$ such, that adjoint action of $hol(\nabla^{inf})$ on $X$ is trivial

2) elements of the image of $\alpha$ are infinitesimal automorphisms

3) $R(X_1,X_2)X=[\kappa(X_1,X_2),X]+\kappa([\alpha(X_1),X],\alpha(X_2))-\kappa([\alpha(X_2),X],\alpha(X_1)).$
\end{prop}
\begin{proof}
The proof of the statements 1) and 2) can be found in \cite{odk10}. The
formula 3) is a direct computation of curvature of induced connection.
\end{proof}

\subsection{Four definitions of symmetric space}

Since the symmetric space can be defined in many different ways, look in
\cite[Chapter IV]{odk2},\cite{odk3}, \cite[Chapter XI]{odk4} and
\cite[Chapter 1]{odk11}, we will review four approaches to them 
and show that they are equivalent.  We denote them by Roman
numbers.

In the case that the homogeneous model $G\to G/H$ is reductive, i.e. if
there is an $Ad(H)$-invariant decomposition
$\mathfrak{g}=\mathfrak{h}+\mathfrak{m}$, there is an isomorphism between
$\mathcal{G}\times_H \mathfrak{m}$ and $TM$.  Then the fundamental
derivative defines the linear connection $\nabla_X=D_s$ on $M$, where $s$ is a
lift of $X$ along $\mathfrak{m}$.  On the other hand any linear connection
$\nabla$ defines a Cartan connection on $P^1M\to M$ of type
$(Gl(n,\mathbb{R})\ltimes_V \mathbb{R}^{n},Gl(n,\mathbb{R}))$, where $V$ is
the standard representation of the linear group, and
$\nabla$ coincides with the fundamental derivative.  These Cartan geometries
are called affine and the homogeneous model is just an affine space, where
the $Gl(n,\mathbb{R})$ acts on $\mathbb{R}^{n}$ by the standard
representation $V$.  To summarize we get the following:

\begin{prop}
There is the bijection between linear connections on $M$ and affine Cartan geometries. The curvature $\kappa$ decomposes to torsion $T$ of $\nabla$ with values in $\mathbb{R}^{n}$ and curvature $R$ of $\nabla$ with values in $Gl(n,\mathbb{R})$.
\end{prop}

\noindent\bf I) Symmetric space as a special Cartan geometry of affine type \rm

\begin{definition}
We say that a torsion-free, affine Cartan geometry with complete infinitesimal automorphisms is a symmetric
space if $D_s\kappa=0$ (or equivalently $\nabla_X R=0$) for any section $s$
of $P^1M\times_{Gl(n,\mathbb{R})} \mathfrak{m}$.
\end{definition}

Following \cite[Chapter II]{odk4}, $H$-structure is a reduction of $P^1M$ to
$H\subset Gl(n,\mathbb{R})$, i.e.  an $H$-principal subbundle
$\mathcal{Q}\subset P^1M$.  The obstruction for the existence of the
reduction is a holonomy, for details look in \cite[Chapter II]{odk4}.  In
the situation of the symmetric spaces the holonomy algebra at $u$ is given
by $\kappa(u)$ and $D_s\kappa=0$ gives that the holonomy does not depend on
base point, i.e.  holonomy algebras at all points are isomorphic.  Then the
reduction theorem can by stated as follows.

\begin{prop}
Let $P^1M\to M$ be a symmetric space and assume that $H\subset
Gl(n,\mathbb{R})$ contains the holonomy group.  Then there is a torsion-free Cartan geometry $Q\to M$ with complete infinitesimal automorphisms of type $(H\ltimes_V
\mathbb{R}^{n},H)$ such, that $D_s\kappa=0$ for every section $s$ of
$\mathcal{Q}\times_{H} \mathfrak{m}$ and the original Cartan connection is
equal to the extension via the inclusion of $H$.  This can be called the
$H$-structure on symmetric space.
\end{prop}
\begin{proof}
First we reduce the geometry to the holonomy subbundle following \cite[Chapter II,
Theorem 7.1]{odk4} and then we extend it to the required geometry via the inclusion of
the holonomy group in $H$.
\end{proof}

In this setting we can define the pseudo-hermitian and para-pseudo-hermitian
symmetric spaces.  We tell some more details on them, because the
construction of symmetric parabolic contact structures will start from them.

\begin{example}
We will always assume that $n=p+q$ in the entire article. If $H\subset Gl(n,\mathbb{C})\cap O(2p,2q)=U(p,q)$, then the symmetric space admitting this $H$-structure is called the pseudo-hermitian symmetric space. If $H\subset (Gl(n,\mathbb{R})\times Gl(n,\mathbb{R}))\cap O(n,n)$, then the symmetric space admitting this $H$-structure is called the para-pseudo-hermitian symmetric space.
\end{example}

\noindent\bf II) Symmetric space as a special reductive Cartan geometry \rm

\begin{definition}
We say that the homogeneous model $G\to G/H$ is a symmetric space if there is an $Ad(H)$-invariant decomposition of $\mathfrak{g}=\mathfrak{h}+\mathfrak{m}$, $\mathfrak{h}=[\mathfrak{m},\mathfrak{m}]$ and the largest normal subgroup of $G$ in $H$ is trivial and there is $h\in H$ acting as $-id$ on $\mathfrak{m}$.
\end{definition}

There are more homogeneous models for the same symmetric space, but we can always pass to effective homogeneous model and in the case that there is not $h\in H$ acting as $-id$ on $\mathfrak{m}$, we can add $\mathbb{Z}_2$ via the involutive automorphism given by the $-id$ on $\mathfrak{m}$.

The correspondence to the previous definition is via the mutation, i.e. extension over identity, which is in this case identity on vector spaces, so does not change the connection, for the details and proof look in \cite[Chapter 5, Proposition 7.2]{odk5}.

\begin{prop}
Let $G\to G/H$ be a symmetric space II), then $H$ is holonomy group and $\omega$ is a Cartan connection of type $(H\ltimes \mathbb{R}^{n},H)$ with curvature $\kappa(u)(X,Y)=[X,Y]$ for $X,Y\in \mathfrak{m}$.

Let $P^1M\to M$ be a symmetric space I). Let $H$ be the holonomy group and $\mathfrak{g}=\mathfrak{h}+\mathbb{R}^{n}$ be a Lie algebra with bracket 
\[ [(h_1+X),(h_2+Y)]=[h_1,h_2]-\kappa(X,Y)+[h_1,Y]-[h_2,X].\]
Then there is a Lie group $G$ acting transitively on $M$ with the Lie algebra $\mathfrak{g}$ such, that $(G\to G/H,\omega)$ is symmetric space II).
\end{prop}

In this setting we can define the simple symmetric spaces. We
say that a symmetric space $G\to G/H$ is simple if $G$ is
simple, or $H$ is simple and $G=H\times H$. Similarly, the 
semisimple symmetric spaces
correspond to semisimple groups $G$ (except they are simple by the previous
definition). Since simply connected
covering of any symmetric space is a symmetric space, the classification 
is up to
discrete phenomena given by Lie algebras.  We can simplify the
classification by the following proposition, the proof can be found in
\cite{odk3}.

\begin{prop}
Semisimple symmetric space is product of simple symmetric spaces.  Let
$\mathfrak{g}=\mathfrak{h}+\mathfrak{m}$ be simple Lie algebra such, that
$[\mathfrak{h},\mathfrak{m}]\subset \mathfrak{m}$ and
$\mathfrak{h}=[\mathfrak{m},\mathfrak{m}]$.  Then there is the unique connected
simple connected simple homogeneous symmetric space $G\to G/H$.
\end{prop}

The statement of the proposition can be extended to $H$-structures on 
symmetric spaces, if any product of $H$-structures again carries an 
$H$-structure.

\begin{example}
Classification of the semisimple pseudo-hermitian and para-pseu\-do-hermitian symmetric spaces is easy, because it can be shown that the complexification of $G$ is for these symmetric spaces $1$-graded, for detailed proof look in \cite[Chapter 3]{odk11}. Looking into the classification of $1$-gradings and real forms in \cite[Chapter 3.2 and Appendix B]{odk6} we get following tables. First table contains simple para-pseudo-hermitian symmetric spaces, where the adjoint representation of the semisimple part of $H$ on $\mathfrak{m}$ is $W+W^*$.

\medskip
\begin{tabular}{|c|c|c|}
\hline
$\mathfrak{g}$ & $\mathfrak{h}$&representation $W$\\
\hline
$\mathfrak{su}(n,n)$&  $\mathfrak{sl}(n,\mathbb{C})+\mathbb{R}$& $\lambda_1\otimes \lambda_{n-1}$\\
\hline
$\mathfrak{sl}(p+q,\mathbb{R})$& $\mathfrak{sl}(p,\mathbb{R})+\mathfrak{sl}(q,\mathbb{R})+\mathbb{R}$&$\lambda_1(p) \otimes \lambda_{n-1}(q)$\\
\hline
$\mathfrak{sl}(p+q,\mathbb{H})$& $\mathfrak{sl}(p,\mathbb{H})+\mathfrak{sl}(q,\mathbb{H})+\mathbb{R}$&$\lambda_1(p) \otimes \lambda_{n-1}(q)$\\
\hline
$\mathfrak{so}(p+1,q+1)$&  $\mathfrak{so}(p,q)+\mathfrak{so}(1,1)$&$\lambda_1$\\
$\mathfrak{so}(n,n)$&  $\mathfrak{sl}(n,\mathbb{R})+\mathbb{R}$&$\lambda_2$\\
\hline
$\mathfrak{sp}(n,n)$&  $\mathfrak{sl}(n,\mathbb{H})+\mathbb{R}$&$2\lambda_1$\\
\hline
$\mathfrak{sp}(2n,\mathbb{R})$& $\mathfrak{sl}(n,\mathbb{R})+\mathbb{R}$&$2\lambda_1$\\
\hline
$\mathfrak{so}^\star(4n)$&  $\mathfrak{sl}(n,\mathbb{H})+\mathbb{R}$&$\lambda_2$\\
\hline
$\mathfrak{e}_6(-26)$& $\mathfrak{so}(9,1)+\mathbb{R}$&$\lambda_4$ \\
$\mathfrak{e}_6(6)$& $\mathfrak{so}(5,5)+\mathbb{R}$&$\lambda_4$\\
\hline
$\mathfrak{e}_7(-25)$& $\mathfrak{e}_6(-26)+\mathbb{R}$&$\lambda_1$\\
$\mathfrak{e}_7(7)$& $\mathfrak{e}_6(6)+\mathbb{R}$&$\lambda_1$\\
\hline
\end{tabular}

\medskip
Second table contains simple pseudo-hermitian symmetric spaces.  The adjoint
representation of the semisimple part of $H$ on $\mathfrak{m}$ is now more
complicated, because now $\mathfrak{g}$ is not $1$-graded.  The
complexification of the adjoint representation of the semisimple part of $H$ on
$\mathfrak{m}$ is $W+W^*$, thus the adjoint representation is $W$ if the
type of the representation is $\mathbb{C}$ or $\mathbb{H}$ or it is
complexification of real representation $W$ if the type is $\mathbb{R}$.

\medskip
\begin{tabular}{|c|c|c|c|}
\hline
$\mathfrak{g}$ & $\mathfrak{h}$ & representation $W$& type of $W$\\
\hline
$\mathfrak{su}(p,q)$& $\mathfrak{su}(p_1,q_1)+\mathfrak{su}(p_2,q_2)+\mathfrak{u}(1)$&$\lambda_1(1) \otimes \lambda_{n-1}(2)$&$\mathbb{C}$\\
\hline
$\mathfrak{sl}(2n,\mathbb{R})$& $\mathfrak{sl}(n,\mathbb{C})+\mathfrak{u}(1)$&$\lambda_1\otimes \lambda_{n-1}$&$\mathbb{C}$\\
\hline
$\mathfrak{sl}(n,\mathbb{H})$& $\mathfrak{sl}(n,\mathbb{C})+\mathfrak{u}(1)$&$\lambda_1\otimes \lambda_{n-1}$&$\mathbb{C}$\\
\hline
$\mathfrak{so}(p+2,q)$&  $\mathfrak{so}(p,q)+\mathfrak{so}(2)$&$\lambda_1$&$\mathbb{R}$\\
$\mathfrak{so}(2p,2q)$&  $\mathfrak{su}(p,q)+\mathfrak{u}(1)$&$\lambda_2$&$\mathbb{C}$\\
\hline
$\mathfrak{sp}(p,q)$&  $\mathfrak{su}(p,q)+\mathfrak{u}(1)$&$2\lambda_1$&$\mathbb{C}$\\
\hline
$\mathfrak{sp}(2n,\mathbb{R})$& $\mathfrak{su}(p,q)+\mathfrak{u}(1)$&$2\lambda_1$&$\mathbb{C}$\\
\hline
$\mathfrak{so}^\star(2n)$&  $\mathfrak{su}(p,q)+\mathfrak{u}(1)$&$\lambda_2$&$\mathbb{C}$\\
$\mathfrak{so}^\star(2n+2)$&  $\mathfrak{so}^\star(2n)+\mathfrak{so}^\star(2)$&$\lambda_1$&$\mathbb{H}$\\
\hline
$\mathfrak{e}_6(-78)$ & $\mathfrak{so}(10)+\mathfrak{u}(1)$&$\lambda_4$&$\mathbb{C}$\\
$\mathfrak{e}_6(-14)$ & $\mathfrak{so}(10)+\mathfrak{u}(1)$&$\lambda_4$&$\mathbb{C}$\\
$\mathfrak{e}_6(-14)$& $\mathfrak{so}(8,2)+\mathfrak{u}(1)$&$\lambda_4$&$\mathbb{C}$\\
$\mathfrak{e}_6(-14)$& $\mathfrak{so}^\star(10)+\mathfrak{u}(1)$&$\lambda_4$&$\mathbb{C}$\\
$\mathfrak{e}_6(2)$& $\mathfrak{so}(6,4)+\mathfrak{u}(1)$&$\lambda_4$&$\mathbb{C}$\\
$\mathfrak{e}_6(2)$& $\mathfrak{so}^\star(10)+\mathfrak{u}(1)$&$\lambda_4$&$\mathbb{C}$\\
\hline
$\mathfrak{e}_7(-133)$ & $\mathfrak{e}_6(-78)+\mathfrak{u}(1)$&$\lambda_1$&$\mathbb{C}$\\
$\mathfrak{e}_7(-25)$ & $\mathfrak{e}_6(-78)+\mathfrak{u}(1)$&$\lambda_1$&$\mathbb{C}$\\
$\mathfrak{e}_7(-25)$& $\mathfrak{e}_6(-14)+\mathfrak{u}(1)$&$\lambda_1$&$\mathbb{C}$\\
$\mathfrak{e}_7(-5)$& $\mathfrak{e}_6(2)+\mathfrak{u}(1)$&$\lambda_1$&$\mathbb{C}$\\
$\mathfrak{e}_7(-5)$& $\mathfrak{e}_6(-14)+\mathfrak{u}(1)$&$\lambda_1$&$\mathbb{C}$\\
$\mathfrak{e}_7(7)$& $\mathfrak{e}_6(2)+\mathfrak{u}(1)$&$\lambda_1$&$\mathbb{C}$\\
\hline
\end{tabular}

\medskip
And finally we list the complex simple symmetric spaces, which have both structures, 
where the adjoint
representation of $H$ is $W+W^*$, and complex representation $W$ is shown 
in the table.

\medskip
\begin{tabular}{|c|c|c|}
\hline
$\mathfrak{g}$ & $\mathfrak{h}$ & representation $W$\\
\hline
$\mathfrak{sl}(p+q,\mathbb{C})$& $\mathfrak{sl}(p,\mathbb{C})+\mathfrak{sl}(q,\mathbb{C})+\mathbb{C}$& $\lambda_1(p) \otimes \lambda_{n-1}(q)$\\
$\mathfrak{so}(p+2,\mathbb{C})$& $\mathfrak{so}(p,\mathbb{C})+\mathfrak{so}(2,\mathbb{C})$& $\lambda_1$\\
$\mathfrak{sp}(n,\mathbb{C})$& $\mathfrak{sl}(2n,\mathbb{C})+\mathbb{C}$&$2\lambda_1$\\
$\mathfrak{so}(2n,\mathbb{C})$& $\mathfrak{sl}(n,\mathbb{C})+\mathbb{C}$& $\lambda_2$\\
$\mathfrak{e}_6$ & $\mathfrak{so}(10,\mathbb{C})+\mathbb{C}$&$\lambda_4$\\
$\mathfrak{e}_7$ & $\mathfrak{e}_6+\mathbb{C}$&$\lambda_1$\\
\hline
\end{tabular}
\end{example}

\medskip
\noindent\bf III) Symmetric spaces as manifolds with a smooth system of symmetries\rm

\begin{definition}
Let $M$ be a connected smooth manifold and $S:M\times M \to M$ a smooth
mapping.  We denote $S(x,y)=S_xy$ and say that $S_x$ is a symmetry at $x$. 
We call $(M,S)$ a symmetric space under the following four conditions:

(A1) $S_xx=x$

(A2) $S_x(S_xy)=y$

(A3) $S_xS(y,z)=S(S_xy,S_xz)$

(A4) $T_xS_x=-\operatorname{id}_{T_xM}$.
\end{definition}

There is the relation between the definitions III) and II) as described in
\cite{odk3}:

\begin{prop}
Let $G\to G/H$ be a homogeneous symmetric space.  Then the linear map
defined as $-\operatorname{id}$ on $\mathfrak{m}$ has unique extension to an involution
$\sigma$ on $G$.  Then we define the symmetries 
\[ S_{fH}gH=f\sigma(f^{-1}g)H.  \]

Let $(M,S)$ be a symmetric space III) and $G\subset \operatorname{Diff}(M)$
the group generated by the symmetries $S_x$. Then $G$ is a Lie group with transitive action on $M$, i.e. $M=G/H$. Thus $G\to G/H$ is a homogeneous symmetric space.
\end{prop}

\noindent\bf IV) Symmetric spaces via geodesic symmetries \rm

Let $\nabla$ be a linear connection, then the geodesic symmetry at $x$ is
the mapping 
\[\operatorname{Exp}_x(X)\mapsto \operatorname{Exp}_x(-X), \]
where $\operatorname{Exp}_x(tX)$ is the geodesic 
of $\nabla$ starting at $x$ in the direction of $X$.

\begin{definition}
Let $M$ be a connected smooth manifold and $\nabla$ a complete affine connection. We say that $(M,\nabla)$ is the affine symmetric space if the geodesic symmetries $S_x$ at all $x$ are affine transformations, i.e. $(S_x)^*\nabla=\nabla$.
\end{definition}

The correspondences between the definitions I), III), and IV) can be found again in
\cite{odk3}:

\begin{prop}
Let $(M,\nabla)$ be an affine symmetric space, then $\nabla R=0$ for $R$ the curvature of $\nabla$. So the corresponding affine Cartan geometry is symmetric.

Let $(M,S)$ be a symmetric space. Then there is the linear connection 
\[ \nabla_XY(x)=\frac12[X,Y+(S_x)^*Y](x)\]
such that $S_x$ is the geodesic symmetry at $x$.
\end{prop}

Next, we discuss the morphisms and show that the categories of symmetric
spaces coming 
from the definitions I)-IV) are equivalent.

\begin{definition} Let us consider the following categories I)--IV):
\begin{itemize}
\item[I)] the morphisms are the morphisms of affine Cartan geometries
\item[II)] the morphisms are the Lie group homomorphisms compatible with the decompositions
\item[III)] the morphisms are maps $f:M\to M'$ such, that $f(S_xy)=S'_{f(x)}f(y)$
\item[IV)] the morphisms are the affine transformations
\end{itemize}
\end{definition}

We shall indicate, why the categories I)--IV) are equivalent, 
for details look in \cite[Chapter 1]{odk11}.

I) $\leftrightarrow$ IV) the morphisms of the affine Cartan geometries are exactly the affine transformations and $P^1$ is a functor.

IV) $\leftrightarrow$ III) for the affine maps $f$ and all geodesic symmetries
we know
$f(S_xy)=S'_{f(x)}f(y)$. The other direction is computed directly 
from the definition.

III) $\leftrightarrow$ II) the axiom (A3) implies 
that the morphisms define Lie group homomorphisms 
and (A4) gives the compatibility with decomposition. 
The other direction is again a direct computation from the definition.

\subsection{Parabolic geometries with smooth system of symmetries}

The symmetric parabolic geometries defined in Definition \ref{symLenka}
provide a generalization of the definition IV) above. 
In this article we will add one more assumption. 
We look at a generalization of definition III):

\begin{definition}
Let us consider a symmetric parabolic geometry $(\mathcal{P}\to M,\omega)$
of type $(L,P)$. We say that $S$ is a smooth system of symmetries, 
if the map $S:M\times M\to M$ is
smooth and satisfies axiom (A3).  We say that system is involutive if (A2)
holds as well.
\end{definition}

These smooth systems of symmetries were investigated in \cite{odk7} for the
$1$-graded parabolic geometries and it was shown that $(M,S)$ is a symmetric
space from definition III).  In the rest of the article we will deal with
the smooth systems of symmetries on parabolic contact geometries and we
construct examples of such geometries.  First examples are the homogeneous
models.

\begin{prop}
For every parabolic subalgebra $\mathfrak p$ of a semisimple $\mathfrak{l}$, there is a 
homogeneous model with smooth system of symmetries.
\end{prop}
\begin{proof}
The construction is based on the Iwasawa decomposition \cite{odk6}[2.3.5] and on the classification of
the parabolic geometries \cite{odk6}[3.2].  Let $\delta$ be a Cartan involution,
$\Delta_r^+$ be the set of positive (restricted for the real case) roots, and
$\Sigma \subset \Delta_r^+$ be the subset of simple positive (restricted) roots
determining $\mathfrak{p}$.  Let
$\mathfrak{l}=\mathfrak{k}+\mathfrak{a}+\mathfrak{n}$ be the Iwasawa
decomposition. If $X=X_0+\sum_{\lambda \in \Delta_r}X_\lambda$ is decomposition of $X$ to the root spaces, then the
corresponding decomposition to $\mathfrak{k}+\mathfrak{a}+\mathfrak{n}$  is 

\[ (X_0\cap \mathfrak{k}+\sum_{\lambda \in
\Delta_r}(X_{-\lambda}+\delta X_{-\lambda}))+(X_0\cap
\mathfrak{a})+(\sum_{\lambda \in \Delta_r}(X_{\lambda}-\delta
X_{-\lambda})).\]

Now it is obvious that $\mathfrak{a}+\mathfrak{n}\subset
\mathfrak{p}$ and $\mathfrak{k}\cap \mathfrak{p}\subset \mathfrak{l}_0$. 
Due to  the Iwasawa decomposition on the group level, 
$K$ acts transitively on
$L/P$ cf.  \cite{odk6}[3.2.4] and \cite{odk6}[3.2.9].  Now we define
the involution $\sigma$ of $\mathfrak{k}$ by 
\[ X_{-\lambda}\mapsto (-1)^{ht(\lambda)}X_{-\lambda}.\]

The involutions $\sigma$ and $\delta$ 
commute and
$L_0$ is contained in the fixed point set of both involutions.

Let $L$ be such, that $\sigma$ is given by $h\in K\cap L_0=H$ (such $L$ 
exists
since $\sigma$ is an involution).  Then we define $S_{gH}fH=ghg^{-1}fH$ for
$g,f\in K$.  It clearly satisfies (A1)--(A3) and the inclusion $K\to G$
implies
that $S$ is covered by automorphisms of flat Cartan geometry.  Due to the
definition of $\sigma$, it acts as $-\operatorname{id}$ on 
$\mathfrak{l}_{-1}$, thus we get
the claim.
\end{proof}

We shall construct further nontrivial examples in the next sections.

\section{General construction of parabolic geometries with smooth systems of symmetries}

In this section we investigate the reflexion spaces, which are another generalizations of symmetric spaces from the definition III) and use them to find general construction of parabolic geometries with smooth systems of symmetries.

\begin{definition}
Let $M$ be a connected smooth manifold and $S:M\times M \to M$ a smooth
mapping.  If $S$ satisfies (A1), (A2) and (A3), then $(M,S)$ is called a
reflexion space.  A smooth mapping $f:M\to M'$ between reflexion spaces is a
morphism if $f(S_xy)=S'_{fx}fy$.
\end{definition}

The reflexion spaces were introduced and investigated by Loos in
\cite{odk8}.  We recall main steps of the investigation and aply the results
on symmetric parabolic geometries.

We denote $\rho_xy=S_yx$, $I(x)=T_xS_x$ and $J(x)=1/2(T_x\rho_x)$. Let
$x(t)$ be a smooth curve satisfying $x(0)=x$ and $x'(0)=X\in T_xM$, then we
obtain differentiating $S_{x(t)}x(t)=x(t)$ and $S_xS_xx(t)=x(t)$, that
$T_xS_xX+T_x\rho_xX=X$ and $(T_xS_x)^2X=X$ i.e.  $I=id_{TM}-2J$ and
$I^2=id_{TM}$.  Thus if we define $T^-M$ and $T^+M$ as eigenspaces of
eigenvalues $-1$ and $1$ of $I$, then $J$ is projection to $T^-M$.

Further we fix arbitrary $e\in M$ and for any $X\in T_eM$ we define the vector
field 
\[ R_e(X)(x)=\frac{1}{2}T_e\rho_{(S_ex)}X.\] 
Let $e(t)$ be a smooth curve satisfying $e(0)=e$ and $e'(0)=X\in T_eM$, then
we obtain differentiating $S_{S_{e(t)}e}S_ex=S_{e(t)}S_eS_{e(t)}S_ex$, that
$T_e\rho_{S_ex}T_e\rho_{e}X=T_e\rho_{S_eS_eS_ex}X+T_eS_eT_eS_eT_e\rho_{S_ex}X=2T_e\rho_{S_ex}X$
i.e.  $R_e(J(e)X)=R_e(X)$.  It will later turn out that the vector fields
$R_e(X)$ represent infinitesimal action of the group generated by all
symmetries.

We define for all $e\in M$ a tensor \[ T(X,Y)(e)=[R_e(X),R_e(Y)](e)\]
and call it torsion of the reflexion space.  



Loos derived the following formula for the torsion evaluated on vector fields, cf. \cite{odk8} Satz 4.5
\[T(X,Y)=J[JX,Y]+J[X,JY]-J[X,Y]-[JX,JY].\]
By the definition of the
torsion, the torsion vanishes if any of its arguments is from $T^+M$ 
since $R_e(X)=0$ for $X\in T^+M$. Thus, 
$-[X,Y]=T(X,Y)\in T^+M$ for $X,Y\in T^-M$ and $-J[X,Y]=T(X,Y)=0$ for $X,Y\in
T^+M$. In particular, $T^+M$ is integrable.

\begin{definition}
We say that a reflexion space $M$ has maximal torsion if vectors of the form $T(X,Y), X,Y\in T^{-}M$ span $T^+M$.
\end{definition}

Let us denote $F_e$ the integral subvariety of $T^+M$ through $e$.
The leaves $F_e$ have the following properties:

\begin{lemma}\label{leaf}
Let $M$ be a reflexion space, then for all points $x$, $y\in M$
\begin{enumerate}
\item $F_x$ is an embedded submanifold,
\item if $x\in F_y$, then $S_y=S_x$,
\item $F_x$ is diffeomorphic to $F_y$. 
\end{enumerate}
\end{lemma}
\begin{proof} These results were proved by Loos:
\begin{enumerate}
\item \cite{odk8} Satz 6.1
\item \cite{odk8} Satz 6.2
\item \cite{odk8} Satz 6.3
\end{enumerate}
\end{proof}

Similarly to the symmetric spaces, Loos defined distinguished 
linear connections on each
reflexion space $M$ which have the given torsion $T$ and keep $J$
covariantly constant. 

First, choosing   an arbitrary torsion-free connection $D_XY$ 
on $M$ and writing
$X=X^++X^-, Y=Y^++Y^-$ for the decomposition of vector fields to
$T^+M\oplus T^-M$, we define the modified
covariant derivative $\nabla_XY$ of vector fields
by the following formula for functions $f: M\to \mathbb{R}$
\begin{equation*}
\begin{split}
(\nabla_XY(x))f&=X(Yf)-Y(R_x(JX)f)-X^+(R_x(JY)f)\\
&+(D_{X^+}Y^+-J(D_{X^+}Y^+))f-X^+(Y^+f).
\end{split}
\end{equation*}

The following lemma provides the requested connections, together with a 
list their properties which we shall 
need. The proof of this Lemma and further properties of the objects in
question can be found in 
\cite{odk8} Chapter 5:

\begin{lemma}\label{lemnabvl}
Let $M$ be a reflexion space and $\nabla$ the above connection. Given $p\in M$, 
let $\nabla^p=1/2(\nabla+(S_p)^*\nabla)$, then:
\begin{enumerate}
\item The linear connection $\nabla^p$ is invariant with respect to $S_p$, its torsion is $T$, and it leaves $J$ covariantly constant. Moreover, the values $\nabla^p_XY$ coincide for all points $p\in M$ if $X\in T^-M$ or $Y\in T^-M$.

\item For each $e\in M$, $F_e$ is totally geodesic submanifold°and $S_xF_e=F_{(S_xe)}$ for all $x\in M$.

\item For every smooth 
curve $\gamma: \mathbb{R}\to M$, $\gamma(0)=e$, 
$\gamma'(0)=X\in T^{-}M$, the following conditions are equivalent
\begin{enumerate}
\item $\gamma$ is geodesic of $\nabla^p$ for some $p\in M$,
\item $\gamma(2t-s)=S_{\gamma(t)}\gamma(s)$,
\item $\gamma$ is integral curve of $R_e(X)$.
\end{enumerate}
If conditions (a)-(c) are fulfilled, 
then the one-parameter subgroup of symmetries corresponding 
to $\gamma$ is $S_{\gamma(t/2)}S_e$.
\end{enumerate}
\end{lemma}
\begin{proof}

\begin{enumerate}
\item The first claim is proved in \cite{odk8} Satz 5.1.

\item The first part of the claim follows from the fact $J$ is parallel constant and the rest is a simple corollary, because the symmetries preserve the splitting $TM=T^-M\oplus T^+M$ in general.

\item The third part claim is proved in \cite{odk8} Satz 5.7.
\end{enumerate}
\end{proof}

Now, we can proceed to  relation with the symmetric parabolic geometries.

\begin{prop}
Let $M$ be connected and $(\mathcal{P}\to M,\omega)$ be a regular parabolic
geometry of type $(L,P)$ with an involutive smooth system of symmetries $S$. Then $(M,S)$ is a reflexion space with maximal torsion.
\end{prop}
\begin{proof}
The system $S$ satisfies axioms (A1), (A2) and (A3) from the definition of
involutive smooth system of symmetries, so $(M,S)$ is a reflexion space. 

Let $N$ be the orbit through $e\in M$ of the group generated by symmetries. Definitely, the points of the form $S_xe=S_xS_ee\in N$ for all $x\in M$. This defines smooth mapping $f: M\to N: f(x)=S_xS_ee=\rho_{S_ee}x$. Clearly $T_ef(X)=2R_e(X)(e)$ for $X\in T^-_eM$, thus $T_ef(T^{-1}_eM)=T^{-1}_eM\subset T_eN$. Since this holds for arbitrary point $e$, the $T^{-1}M$ is a subdistribution of $TN$. The regularity of $\omega$ and lemma \ref{lemreg1} imply, that $T^{-1}M$ generates the whole $TM$ by the Lie bracket. Thus $TM\subset TN$ and consequently $M=N$.

We learned in lemma \ref{lemnabvl}, claim (3), that the integral curves of $R_e(X)$ are the one parameter subgroups of symmetries. This means that, $R_e(X)$ is a projection of an
infinitesimal automorphism of the parabolic geometry. Moreover, $R_e(X)$ for $X\in T^-M$ generate the entire $TM$, see above. 

Finally, let us remind the result in \cite{odk8} Satz 4.2 saying that $$[[R_e(X),R_e(Y)],R_e(Z)]=R_e([[R_e(X),R_e(Y)],R_e(Z)](e))$$ for all $X,Y,Z\in TM$. Thus, each bracket of even number of infinitesimal automorphisms can be in fact expressed  as a single bracket of two arguments, i.e. via the torsion $T$. In particular, the torsion generates $T^+M$.
\end{proof}

So we are interested in the structure of reflexion spaces with maximal torsion.

\begin{prop}
Let $M$ be reflexion space of maximal torsion, let $G$ be the group
generated by the symmetries, let $H$ be subgroup fixing $F_e$ for a given point $e$, let $K\subset
H$ be the subgroup fixing $e$, and let $h=S_e$, then:
\begin{enumerate}
\item If $M'$ is the leaf space of all $F_p$, then $M'$ is the symmetric space $G/H$
with $S_{fH}gH=fhf^{-1}gH$.

\item The projection $p:M\to M'$ is a morphism of reflexion spaces, $M=G/K$, $e=eK$, and $S_{fe}ge=fhf^{-1}ge$.

\item $G$ over $ M$ with the Maurer-Cartan form is the flat Cartan geometry of type
$(G,K)$ with an involutive smooth system of symmetries $S$ and it is
the correspondence space to the symmetric space $G/H$.
\end{enumerate}
\end{prop}
\begin{proof}
\begin{enumerate}
\item Properties from lemma \ref{leaf}, claims (1,3) are sufficient for the leaf
space $M'$ having the structure of smooth manifold, such that the 
projection $p:
M\to M'$ is a submersion.  Further, lemma \ref{lemnabvl} (2) and lemma 
\ref{leaf}
(2) show that symmetries descend to symmetries on $M'$.  Since (A1), (A2)
and (A3) are still satisfied and also (A4) holds, we get the requested claim from
equivalence of definitions III) and II) of symmetric spaces.

\item Lemma \ref{lemnabvl} (3) implies, that $G$ is a Lie transformation group
of $M$ with Lie algebra generated by vector fields $R_e$. Consequently the maximality of the
torsion ensures, that vector fields $R_e$ and their brackets generate $T_eM$. 
Thus $G$ acts transitively and the claim follows.

\item The symmetries act by left multiplication of elements of $G$ and so they are morphisms of this Cartan
geometry according to the proposition \ref{authomod}. 
\end{enumerate}
\end{proof}

If we combine the two latter propositions, we arrive at the principal result describing general construction of parabolic geometries with involutive smooth system of symmetries:

\begin{theorem}\label{them1}
Let $M$ be connected and $(\mathcal{P}\to M,\omega)$ be a regular parabolic
geometry of type $(L,P)$ with an involutive smooth system of symmetries $S$. 
Then it is the homogeneous Cartan geometry $M=G/K$, where $G$ is the group
generated by symmetries.  Thus all regular parabolic geometries with
involutive smooth system of symmetries are extensions of reflexion spaces
$G/K$.
\end{theorem}
\begin{proof}
We have already deduced that $M=G/K$ follows from the previous two propositions, but then the theorem follows
from proposition \ref{them2}.
\end{proof}

\section{Construction of parabolic contact geometries with smooth system of symmetries}

In this section we investigate the construction in the case of parabolic contact geometries. 

So let $G/K$ be a reflexion space with an underlying symmetric space (II) $G/H$ with the element $h\in H$ defining the symmetries and $(\alpha,i)$ be an extension of homogeneous model $G\to G/K$ to parabolic contact geometry of type $(L,P)$. Let $\mathfrak{g}=\mathfrak{k}+\mathfrak{h/k}+\mathfrak{m}$ be decomposition to $\pm 1$ eigenspaces. Notice $Ad(i(h))^2=Ad(i(h^2))=id$ and the corresponding decomposition $\mathfrak{l}=\alpha(\mathfrak{h/k})+\alpha(\mathfrak{m})+
\alpha(\mathfrak{k})+\mathfrak{f}^++\mathfrak{f}^-$ to $\pm 1$ eigenspaces, where $\mathfrak{f}^++\mathfrak{f}^-$ are the remaining parts of the eigenspaces outside the image of $\alpha$. Thus $i(h)\in L_0$, where $L_0$ is Levi subgroup of $P$ with Lie algebra $\mathfrak{l}_0$.

In the following theorem, we characterize all regular parabolic contact geometries with smooth system of symmetries and semisimple group of symmetries.

\begin{theorem}\label{them3}
Let $G/K$ be a reflexion space with an underlying semisimple symmetric space $G/H$ and $(\alpha,i)$ be an extension to parabolic contact geometry of type $(L,P)$ of dimension $2n+1$. Then

1) $Ad(i(h))$ acts as $(-1)^{i}$ on $\mathfrak{l}_{i}$, i.e. it is an extension to a parabolic contact geometry with involutive smooth system of symmetries.

2) $\alpha(\mathfrak{m})\subset \mathfrak{l}_{-1}+\mathfrak{l}_{1}$, $dim(G/H)=2n$ and $dim(H/K)=1$.

3) $\alpha$ restricted to $\mathfrak{h}$ is a Lie algebra homomorphism

4) $i(K)\subset L_0$, the geometry $G/K$ is reductive and $\mathfrak{h/k}$ is in center of $\mathfrak{h}$. There is a well-defined map $\alpha^{-1}$ from $\mathfrak{l}_{-1}+\mathfrak{l}_{-2}$ to $\mathfrak{h/k}+\mathfrak{m}$ given by partially inverting $\alpha$.

5) If the extension is a regular parabolic geometry, then $G/H$ has got only pseudo-hermitian or para-pseudo-hermitian simple factors.
\end{theorem}
\begin{proof}
Since $\mathfrak{l}$ is simple, there are no simple ideals of $\mathfrak{l}$ in $\mathfrak{l}_0$ or $\mathfrak{l}$. Then employing the brackets we get, that the action of $Ad(i(h))$ is $(-1)^{i}$ on $\mathfrak{l}_{i}$, thus we get 1) and 2).

We know that $\mathfrak{l}_{-2}$ is one dimensional and $\alpha(\mathfrak{h})\subset \mathfrak{l}_{-2}+\mathfrak{l}_{0}+\mathfrak{l}_{2}$. Since $\kappa(\mathfrak{l}_{-2},\mathfrak{l}_{-2})=0$, this already proves 3).

Now, $\alpha(\mathfrak{k})\subset \mathfrak{l}_{0}+\mathfrak{l}_{2}$. For each parabolic contact geometry, $\mathfrak{l}_{-2}+\mathfrak{l}_{2}$ generates subalgebra $\mathfrak{z}$ isomorphic to $\mathfrak{sl}(2,\mathbb{R})$ or $\mathfrak{su}(2)$, and these are the only parts of $\alpha(\mathfrak{h})\subset \mathfrak{l}_{-2}+\mathfrak{l}_{0}+\mathfrak{l}_{2}$ with nontrivial action on $\mathfrak{l}_{-2}$. Since $\mathfrak{g}$ is semisimple, $\mathfrak{h}$ contains only semisimple or abelian simple factors. We investigate all possible cases of $\alpha(\mathfrak{h})\cap \mathfrak{z}$:

a) $\alpha(\mathfrak{h})\cap \mathfrak{z}$ is nilpotent, then 4) holds.

b) $\alpha(\mathfrak{h})\cap \mathfrak{z}=\mathfrak{z}$. Thus preimage of $\mathfrak{z}$ contains subalgebra isomorphic to $\mathfrak{z}$. Then since $\mathfrak{z}$ is not factor of $\mathfrak{g}$, the root space in $\mathfrak{z}\cap \mathfrak{k}$ has nontrivial action on $\mathfrak{m}$ and its image in $\mathfrak{l}_{2}$ has trivial action. Contradiction.

c) $\alpha(\mathfrak{h})\cap \mathfrak{z}$ is solvable. Since $\mathfrak{h}$ does not contain solvable factors, there is subalgebra of $\mathfrak{h}$ isomorphic to $\mathfrak{z}$ with an solvable subalgebra mapped onto $\alpha(\mathfrak{h})\cap \mathfrak{z}$. The image of root space in $\mathfrak{z}\cap \mathfrak{k}$ maps $\mathfrak{l}_{-2}$ to $\mathfrak{l}_{0}$. Contradiction.

Thus, if the extension is a regular parabolic geometry, then each simple factor in $G/H$ has a non-trivial center. From the classification of the semisimple symmetric spaces with non-trivial center of $\mathfrak{h}$ we know, that only those in 5) may occur.
\end{proof}

Now, the $Ad(i(h))$-action restricts the curvature of the extension in the following way.

\begin{lemma}\label{norcur}
Let $G/K$ be a reflexion space with underlying semisimple symmetric space $G/H$ and $(\alpha,i)$ extension to contact parabolic geometry of type $(L,P)$ of dimension $2n+1$. Then

1) $\kappa(X,Y)=[\alpha(\alpha^{-1}(X)),\alpha(\alpha^{-1}(Y))]-\alpha([\alpha^{-1}(X),\alpha^{-1}(Y)])$ for $X,Y\in \mathfrak{l}_{-2}+\mathfrak{l}_{-1}$.

2) $\kappa(\mathfrak{l}_{-1},\mathfrak{l}_{-1})\subset \mathfrak{l}_{-2}+\mathfrak{l}_{0}+\mathfrak{l}_{2}$, $\kappa(\mathfrak{l}_{-1},\mathfrak{l}_{-2})\subset \mathfrak{l}_{-1}+\mathfrak{l}_{1}$, $\kappa(\mathfrak{l}_{-2},\mathfrak{l}_{-2})=0$.

3) If the underlying symmetric space is simple, then the extension is regular if and only if the $\mathfrak{l}_{-2}$ part of $\alpha$ is given by the bracket on $\mathfrak{l}_{-1}$.

4) The extension to regular normal parabolic geometry is always torsion free, and the normality conditions are $\sum_i [Z_i,\kappa(X^{-1},X_i)]=0$ and $\sum_i \kappa([Z_i,X^{-2}],X_i)=0$ for any $X=X^{-1}+X^{-2}\in \mathfrak{l}_{-1}+\mathfrak{l}_{-2}$, where $X_i$ is basis of $\mathfrak{l}/\mathfrak{p}$ and $Z_i$ dual basis to $X_i$.
\end{lemma}
\begin{proof}
1) The curvature depends only on the class in $\mathfrak{l}/\mathfrak{p}$ and because $\alpha\circ \alpha^{-1}=id_{\mathfrak{l}/\mathfrak{p}}$ the claim 1) holds.

2) The claim is consequence of a direct computation with $Ad(i(h))$-action on formula in 1).

3) From 2) we get that $\mathfrak{l}_{-2}$ part of the curvature depends only on $\mathfrak{l}_{-1}$, then regularity is equivalent to the condition that the $\mathfrak{l}_{-2}$ part of 
\[ [\alpha(\alpha^{-1}(X)),\alpha(\alpha^{-1}(Y))]-\alpha([\alpha^{-1}(X),\alpha^{-1}(Y)])=0 \]
for $X,Y\in \mathfrak{l}_{-1}$. But for $X,Y\in \mathfrak{l}_{-1}$, the $\mathfrak{l}_{-2}$ part of $[\alpha(\alpha^{-1}(X)),\alpha(\alpha^{-1}(Y))]$ is $[X,Y]$ and the part of $\alpha([\alpha^{-1}(X),\alpha^{-1}(Y)])$ is multiple of the projection to the center of $[\alpha^{-1}(X),\alpha^{-1}(Y)]$, i.e. the image of the center is given by the bracket.

4) The normality conditions can be written, in terms of $X_i$ basis of $\mathfrak{l}/\mathfrak{p}$ and $Z_i$ dual basis of $\mathfrak{p}_+$ and $X\in \mathfrak{l}/\mathfrak{p}$, as 
\[\sum_i \kappa([Z_i,X],X_i)=2\sum_i [Z_i,\kappa(X,X_i)])].\]
So for $X_i\in \mathfrak{l}_{-2}$ we obtain $[Z_i,X]\in \mathfrak{p}$ and $\kappa([Z_i,X],X_i)=0$. For $X_i\in \mathfrak{l}_{-1}$ and for $X \in \mathfrak{l}_{-1}$ we get $\kappa([Z_i,X],X_i)=0$. So $\sum_i \kappa([Z_i,X],X_i)\in \mathfrak{l}_{0}+\mathfrak{l}_{-2}$. For $X_i\in \mathfrak{l}_{-2}$ we obtain $[Z_i,\kappa(X,X_i)]\in \mathfrak{l}_{1}$. For $X_i\in \mathfrak{l}_{-1}$ and $X\in \mathfrak{l}_{-2}$ we get $[Z_i,\kappa(X,X_i)]\in \mathfrak{l}_{0}+\mathfrak{l}_{2}$. For $X_i\in \mathfrak{l}_{-1}$ and $X\in \mathfrak{l}_{-1}$ we get $[Z_i,\kappa(X,X_i)]\in \mathfrak{l}_{1}$. Since the homogeneous components of the torsion vanish, so does the whole torsion. Then normality conditions looks like as in the proposition.
\end{proof}

{}From the normality we get torsion-freeness, thus most of normal symmetric parabolic contact geometries are immediately locally flat.

\begin{prop}
Extensions to regular normal parabolic contact geometries with involutive smooth systems of symmetries and non-trivial curvature have $\mathfrak{l}$ equal to $\mathfrak{sl}(n,\mathbb{R})$, $\mathfrak{su}(p,q)$ or $\mathfrak{sp}(2n,\mathbb{R})$.
\end{prop}
\begin{proof}
Looking in the table of parabolic contact geometries, the geometries having harmonic curvature of non-torsion type are those in the proposition.
\end{proof}

So we are interested in construction of such geometries. The construction starting with simple pseudo-hermitian or para-pseudo-hermitian symmetric space, which is not complex, is summarized in the following proposition.

\begin{prop}\label{pp29}
Let $G/H$ be simple pseudo-hermitian or para-pseudo-hermitian symmetric space of dimension $2n$, which is not complex. Let $h\in K \subset H$ be the subgroup of dimension $dim(H)-1$, whose Lie algebra contains semisimple part of $\mathfrak{h}$. Let $i:K\to L_0$ be an injective homomorphism. Assume that the adjoint representations of $K$ on $\mathfrak{m}$ and $i(K)$ on $\mathfrak{l}_{-1}$ are isomorphic. Let $\alpha$ have the following components:

1) $i'$ on $\mathfrak{k}$ with values in $\mathfrak{l}_{0}$

2) induced by the isomorphism of adjoint representations on $\mathfrak{m}$ with values in $\mathfrak{l}_{-1}$ and induced by some morphism of adjoint representations on $\mathfrak{m}$ with values in $\mathfrak{l}_{1}$

3) induced by the bracket on $\mathfrak{h/k}$ with values in $\mathfrak{l}_{-2}$, while $\alpha$ is arbitrary on $\mathfrak{h/k}$ with values in $\mathfrak{l}_{2}$ or in the centralizer of $i'(\mathfrak{k})$ in $\mathfrak{l}_{0}$

Then $(\alpha,i)$ is an extension to regular parabolic contact geometry of type $(L,P)$ of dimension $2n+1$ with invariant smooth system of symmetries and all extensions $\alpha$ (for fixed $i$) are of this form.
\end{prop}
\begin{proof}
Since the decomposition $\mathfrak{g}=\mathfrak{k}+\mathfrak{h/k}+\mathfrak{m}$ is $Ad(K)$-invariant the $\alpha$ is well-defined. The conditions (ii) and (iii) from definition for $\alpha$ to be extension hold by definition of $\alpha$; (i) holds, because the adjoint representations are identified by $\alpha$. Defining $\alpha$ in another way breaks some of the conditions (i)-(iii). The regularity comes from 3).
\end{proof}

The complex case is more complicated, since the semisimple part of $H$ has dimension $dim(H)-2$. Thus one has to choose one dimension subgroup of center of $H$ to get $K$ and regularity impose one more condition on possible morphism of adjoint representations.

We are interested in all possible extensions that are regular, normal and non-isomorphic. The regularity follows from previous proposition. By the general theory (Theorem 3.1.14. in \cite{odk6}), there is always a normal Cartan connection enjoying the same automorphisms as the given parabolic geometry. Thus, without loss of generality, we shall restrict our attention to extensions $\alpha$ leading directly to normal geometries. The automorphisms of the Cartan connection can be computed from proposition \ref{lab_1} and there are the following morphisms in the class of all possible extensions:

\begin{lemma}
All morphisms of bundles $G\times_i P$ between extended geometries from previous proposition are generated by those of the following form:

M1) $(g,p)\mapsto (g,Ap)$, where $A\in P$ is such, that $Ak=kA$ for $k\in i(K)$. Then the pullback of $\omega_\alpha$ is $\omega_{Ad(A^{-1})\circ \alpha}$

M2) $(g,p)\mapsto (gB,p)$ where $B$ is in center of $H$. Then the pullback of $\omega_\alpha$ is $\omega_{\alpha \circ Ad(B^{-1})}$

M3) $(g,p)\mapsto (\phi(g),p)$, where $\phi$ is Lie group automorphism of $G$ such, that $\phi(k)=k$ for $k\in K$. Then the pullback of $\omega_\alpha$ is $\omega_{\alpha \circ T\phi}$
\end{lemma}
\begin{proof}
Since $A$, $B$ and $\phi$ commute with elements of $K$ and $i(K)$, the morphisms are well defined. Then it is easy to compute the pullbacks.

Let $\Phi: G\times_i P\to G\times_i P$ be morphism of bundles such, that $(\Phi)^*\omega_\alpha=\omega_{\alpha'}$ and $\Phi(e,e)=(B,A)$. Then $(\Phi)^*(\alpha \circ \omega)(X)(e,e)=\alpha \circ \omega (T\Phi (X))(B,A)=Ad(A^{-1})\circ \alpha \circ \omega(Tr^{A^{-1}} \circ Tl^{B^{-1}}\circ T\Phi (X))(e,e)=\alpha' \circ \omega(X)(e,e)$. Since $\omega_\alpha(X)(g',k')=Ad(k'^{-1})\circ \alpha \circ \omega(Tr^{k'^{-1}} \circ Tl^{g'^{-1}}(X))(e,e)$, the $\Phi$ is uniquely determined by $\Phi(e,e)$ and $T_{(e,e)}\Phi$. Thus only possible choices with nontrivial action are obtained from M1), M2) and M3).
\end{proof}

In the rest of the paper, we will describe all non flat examples with simple group generated by symmetries. We investigate following questions:

\bf 1) existence of extension and how all possible extensions look like? \rm

We investigate, which non-complex simple pseudo-hermitian or para-pseudo-hermitian symmetric spaces satisfies the conditions of previous proposition. In particular, we find $K$ and $i$ as in proposition \ref{pp29}.

\bf 2) which extensions are normal? \rm

We compute normality conditions from lemma \ref{norcur} and solve them using Maple.

\bf 3) which extensions are the same, i.e. differ by an automorphism? \rm

We follow the proposition \ref{lab_1} to determine the infinitesimal automorphisms. We know that each element in the image of $\alpha$ induces an infinitesimal automorphism. In the non flat case, we compute using Maple, that there are no other ones.

\bf 4) which extensions are equivalent and determine the equivalence classes? \rm

The only possible morphisms are M1), M2) and M3) from previous lemma. We use Maple to compute the action of them on $\omega$ and choose suitable representants of equivalence classes of extensions.

\section{Extensions to parabolic contact structures of dimension $3$}

We treat the dimension $3$ separately, because on both sides of parabolic contact geometries and symmetric spaces exceptional phenomena arise.

There are only two types of simple symmetric spaces of dimension two to start with, i.e. $\mathfrak{so}(3)/\mathfrak{so}(2)$ and $\mathfrak{so}(2,1)/\mathfrak{so}(1,1)$. Thus $K$ is discrete in this situation, i.e. $K\cong \mathbb{Z}_2$ consists only of the symmetry $h$.

The parabolic contact structures of dimension $3$ we are interested in, are those having $\mathfrak{l}$ one of $\mathfrak{sl}(3,\mathbb{R})$, $\mathfrak{su}(2,1)$ and $\mathfrak{sp}(4,\mathbb{R})$.

\begin{lemma}
For any choice of $\mathfrak{l}$ and symmetric space $\mathfrak{so}(3)/\mathfrak{so}(2)$ or $\mathfrak{so}(2,1)/\mathfrak{so}(1,1)$ there is $i$ satisfying assumptions of proposition \ref{pp29}. If there is no normal subgroup of $L$ in $P$, then the $i$ is unique up to equivalence.
\end{lemma}
\begin{proof}
To define $i: K \to L_0$ for the extensions, it suffices to give the image of $h$, which will be unique if there is no normal subgroup of $L$ in $P$. We map $h$ to element 
\[ \left( \begin{array}{ccc}
-1 & 0 & 0 \\
0 & 1 & 0 \\
0 & 0 & -1 \end{array} \right)\]
 in $L_0$ for $\mathfrak{sl}(3,\mathbb{R})$, $\mathfrak{su}(2,1)$ and map $h$ to element 
\[ \left( \begin{array}{cccc}
-1 & 0 & 0 & 0 \\
0 & 1 & 0& 0 \\
0 & 0 & 1 & 0\\
0& 0& 0& -1 \end{array} \right)\]
for $\mathfrak{sp}(4,\mathbb{R})$.

Any linear isomorphism is isomorphism of representations $K$ and $i(K)$, thus $i$ satisfies assumptions of proposition \ref{pp29}.

If there is no normal subgroup of $L$ in $P$, then the $i$ is unique up to equivalence.
\end{proof}

Then following the proposition \ref{pp29} we can construct $\alpha$ as follows. First, we write $(e,x_1,x_2)$ for the following matrices
\[ \left( \begin{array}{ccc}
0 & e & -x_1  \\
-c\cdot e & 0 & -c\cdot x_2  \\
x_1 & x_2 & 0    \end{array} \right)\]
in $\mathfrak{so}(2+c,1-c)$. Further, $b_1, b_2, b_3, b_4, a_1, a_2, c_1, d_1, d_2, d_3, d_4$ are real numbers such, that $b_1b_4-b_2b_3\neq 0$.

For $\mathfrak{l}=\mathfrak{sl}(3,\mathbb{R})$, the proposition \ref{pp29} implies
\[ \alpha(e,x_1,x_2)=
\left( \begin{array}{ccc}
a_1e & d_1x_1+d_2x_2 & c_1e \\
b_1x_1+b_2x_2 & a_2e  & d_3x_1+d_4x_2 \\
(b_1b_4-b_2b_3)e & b_3x_1+b_4x_2 & -(a_1+a_2)e \end{array} \right).\] 

Similarly, in the case $\mathfrak{l}=\mathfrak{su}(2,1)$ 
\[\alpha(e,x_1,x_2)=\]
\[
\left( \begin{array}{ccc}
a_1e+a_2ei & * & c_1ei \\
b_1x_1+b_2x_2+(b_3x_1+b_4x_2)i & -2a_2ei  & d_1x_1+d_2x_2+(d_3x_1+d_4x_2)i \\
2(b_1b_4-b_2b_3)ei & * & -a_1e+a_2ei \end{array} \right),\] 
where $*$ means that, the entry is determined by the structure of $\mathfrak{su}(2,1)$.

For $\mathfrak{l}=\mathfrak{sp}(4,\mathbb{R})$,
\begin{multline*}
\alpha(e,x_1,x_2)=\\
\left( \begin{array}{cccc}
a_1e & d_1x_1+d_2x_2& d_3x_1+d_4x_2 & c_1e \\
b_1x_1+b_2x_2 & a_2e & a3e & d_3x_1+d_4x_2 \\
b_3x_1+b_4x_2 & a4e &-a_2e  & -d_1x_1-d_2x_2 \\
2(b_1b_4-b_2b_3)e & b_3x_1+b_4x_2& -b_1x_1-b_2x_2 & -a_1e \end{array} \right).
\end{multline*}

We skip computations of normality conditions and automorphisms, which can be easily done due to the dimension. But we look on equivalence classes of extensions in detail. Following the general strategy, we shall employ the morphism of types M1), M2) and M3) to construct suitable canonical forms of the morphisms $\alpha$, and thus we shall classify all equivalence classes of $\alpha$ for fixed $i$.

In $\mathfrak{l}=\mathfrak{sl}(3,\mathbb{R})$ case we can use morphisms of type M1) to get
$$(b_1b_4-b_2b_3)'=\frac{(b_1b_4-b_2b_3)}{n_2^2n_3},\ b_1'=\frac{b_1n_3}{n_2},\ b_2'=\frac{b_2n_3}{n_2},\ b_3'=\frac{b_3}{n_3^2n_2},\ b_4'=\frac{b_4}{n_3^2n_2},$$
so one can choose $b_1b_4-b_2b_3=1$ and one of $b_1, b_2, b_3, b_4=1$.

In the $c=1$ case we can use morphisms of type M2) to get
\[b_1'=b_1\cdot cos(n_1)-b_2\cdot sin(n_1),\]
\[b_2'=b_1\cdot sin(n_1)+b_2\cdot cos(n_1),\]
\[b_3'=b_3\cdot cos(n_1)-b_4\cdot sin(n_1),\]
\[ b_4'=b_3\cdot sin(n_1)+b_4\cdot cos(n_1),\]
so we can choose $b_2=0$, and then $t:=b_3'=\frac{b_1b_3+b_2b_4}{b_1b_4-b_2b_3}$. Finally, using morphisms of type M3) and M1) we can change $b_3'=-b_3$.

In the $c=-1$ case we can use morphisms of type M2) to get
\[b_1'=b_1\cdot cosh(n_1)-b_2\cdot sinh(n_1),\]
\[ b_2'=-b_1\cdot sinh(n_1)+b_2\cdot cosh(n_1)\]
\[b_3'=b_3\cdot cosh(n_1)-b_4\cdot sinh(n_1),\]
\[ b_4'=-b_3\cdot sinh(n_1)+b_4\cdot cosh(n_1).\]

Since we can use morphisms of type M2) to exchange $b_1, b_3$ with $b_2, b_4$, we can choose $b_1^2\geq b_2^2$. If $b_1^2>b_2^2$, then we can choose $b_2=0$, and then $t:=b_3'=\frac{-b_1b_3+b_2b_4}{b_1b_4-b_2b_3}$. If $b_1^2=b_2^2$, then we can choose $b_3^2\leq b_4^2$, if $b_3^2<b_4^2$, then we can choose $b_1=1, b_3=0$, if $b_3^2=b_4^2$, then we can choose $b_1=1, b_2=1, b_3=-1, b_4=1$. Again we can get $b_3'=-b_3$, if we use morphisms of type M3) and M1).

\begin{theorem}
Up to equivalences, all regular normal extensions for  $\mathfrak{so}(3)/\mathfrak{so}(2)$ to $\mathfrak{sl}(3,\mathbb{R})$ are given by the following one parameter classes with $t\geq 0$:
\[\alpha(e,x_1,x_2)=
\left( \begin{array}{ccc}
\frac{t}{4}e & -\frac{3t^2+4}{4}x_1+\frac{t}{4}x_2 & -\frac{15t^2+16}{16}e \\
x_1 & -\frac{t}{2}e  & -\frac{3t}{4}x_1-x_2 \\
e & tx_1+x_2 & \frac{t}{4}e \end{array} \right)\] 
with curvature
\[ \kappa((e,x_1,x_2),(h,y_1,y_2))=\]
\[
\left( \begin{array}{ccc}
0 & \frac{3(t^3+t)}{2}(hx_1-ey_1) & 0 \\
0 & 0  & -\frac{3t^2}{2}(hx_1-ey_1)-\frac{3t}{2}(hx_2-ey_2) \\
0 & 0 & 0 \end{array} \right).\]

Up to equivalences, all regular normal extensions for $\mathfrak{so}(2,1)/\mathfrak{so}(1,1)$ to $\mathfrak{sl}(3,\mathbb{R})$ are given by the following one parameter classes:

a) for $b_1^2>b_2^2$, there is one parameter class for $t\geq 0$
\[\alpha(e,x_1,x_2)=
\left( \begin{array}{ccc}
-\frac{t}{4}e & \frac{3t^2-4}{4}x_1-\frac{t}{4}x_2 & \frac{16-15t^2}{16}e \\
x_1 & \frac{t}{2}e  & \frac{3t}{4}x_1+x_2 \\
e & tx_1+x_2 & -\frac{t}{4}e \end{array} \right)\] 
with curvature
\[ \kappa((e,x_1,x_2),(h,y_1,y_2))=\]
\[
\left( \begin{array}{ccc}
0 & \frac{3(t^3-t)}{2}(hx_1-ey_1) & 0 \\
0 & 0  & -\frac{3t^2}{2}(hx_1-ey_1)-\frac{3t}{2}(hx_2-ey_2) \\
0 & 0 & 0 \end{array} \right);\]

b) for $b_1^2=b_2^2$ and $b_3^2<b_4^2$
\[\alpha(e,x_1,x_2)=
\left( \begin{array}{ccc}
\frac14e & -x_1-\frac34x_2 & \frac{1}{16}e \\
x_1+x_2 & -\frac12e  & \frac14x_1+\frac14x_2 \\
e & x_2 & \frac14e \end{array} \right)\] 
with curvature
\[ \kappa((e,x_1,x_2),(h,y_1,y_2))=\]
\[
\left( \begin{array}{ccc}
0 & \frac32(hx_1-ey_1)+\frac32(hx_2-ey_2) & 0 \\
0 & 0  & 0 \\
0 & 0 & 0 \end{array} \right);\]

c) for $b_1^2=b_2^2$ and $b_3^2=b_4^2$
\[\alpha(e,x_1,x_2)=
\left( \begin{array}{ccc}
\frac14e & -\frac18x_1+\frac18x_2 & \frac{1}{32}e \\
x_1+x_2 & -\frac12e  & \frac18x_1+\frac18x_2 \\
2e & -x_1+x_2 & \frac14e \end{array} \right)\] 
with is flat.
\end{theorem}

In $\mathfrak{l}=\mathfrak{su}(2,1)$ case we can use morphisms of type M3) to get $b_1b_4-b_2b_3>0$ and of type M1) to get
\[ (b_1b_4-b_2b_3)'=2(b_1b_4-b_2b_3)(cosh(n_2)+sinh(n_2))^2,\]
so we can choose $b_1b_4-b_2b_3=1$. The actions of M1) and M2) are quite complicated, so we won't state them explicitly, but using morphisms of type M1), M2) and M3) we can get $b_2'=b_3'=0, b_1'=t,b_4'=\frac{1}{t}$, where $$t:=\sqrt{\frac{s+c\sqrt{s^2-4c}}{2c}},\ s=\frac{cb_1^2+b_2^2+cb_3^2+b_4^2}{b_1b_4-b_2b_3}.$$

\begin{theorem}
Up to equivalences, all regular normal extensions for $\mathfrak{so}(3)/\mathfrak{so}(2)$ to $\mathfrak{su}(2,1)$ are given by the following one parameter classes for $s\geq 2$:
\[\alpha(e,x_1,x_2)=
\left( \begin{array}{ccc}
\frac{1+t^4}{8t^2}ie & * & \frac{-(15t^8-34t^4+15)}{128t^4}ie \\
t x_1+\frac{i}{t}x_2 & -\frac{1+t^4}{4t^2}ie  & \frac{-3t^4+5}{16t}x_1+\frac{5t^4-3}{16t^3}ix_2 \\
2ie & * & \frac{1+t^4}{8t^2}ie \end{array} \right),\] 
where $*$ means that, the entry is determined by the structure of $\mathfrak{su}(2,1)$, with curvature 
\[ \kappa((e,x_1,x_2),(h,y_1,y_2))=\]
\[
\left( \begin{array}{ccc}
0 & * & 0 \\
0 & 0  & \frac{3(1-t^8)}{16t^5}(hx_2-ey_2)+\frac{3(1-t^8)}{16t^3}i(hx_1-ey_1)\\
0 & 0 & 0 \end{array} \right),\]
where $*$ means that, the entry is determined by the structure of $\mathfrak{su}(2,1)$.

Up to equivalences, all regular normal extensions for $\mathfrak{so}(2,1)/\mathfrak{so}(1,1)$ to $\mathfrak{su}(2,1)$ are given by the following one parameter classes for $s>-2$:
\[\alpha(e,x_1,x_2)=
\left( \begin{array}{ccc}
\frac{1-t^4}{8t^2}ie & * & \frac{-(15t^8+34t^4+15)}{128t^4}ie \\
t x_1+\frac{i}{t}x_2 & -\frac{1-t^4}{4t^2}ie  & \frac{3t^4+5}{16t}x_1+\frac{-5t^4-3}{16t^3}ix_2 \\
2ie & * & \frac{1-t^4}{8t^2}ie \end{array} \right),\] 
where $*$ means that, the entry is determined by the structure of $\mathfrak{su}(2,1)$, with curvature 
\[ \kappa((e,x_1,x_2),(h,y_1,y_2))=\]
\[
\left( \begin{array}{ccc}
0 & * & 0 \\
0 & 0  & \frac{3(1-t^8)}{16t^5}(hx_2-ey_2)+\frac{3(1-t^8)}{16t^3}i(hx_1-ey_1)\\
0 & 0 & 0 \end{array} \right),\]
where $*$ means that, the entry is determined by the structure of $\mathfrak{su}(2,1)$.
\end{theorem}

The $\mathfrak{l}=\mathfrak{sp}(4,\mathbb{R})$ case is flat and all $\alpha$ are equivalent.

\begin{theorem}
Up to equivalence, there is the unique regular normal extension for $\mathfrak{so}(2,1)/\mathfrak{so}(1,1)$ to $\mathfrak{sp}(4,\mathbb{R})$ with

\[\alpha(e,x_1,x_2)=
\left( \begin{array}{cccc}
0 & -\frac14x_1& \frac14x_2 & \frac1/8e \\
x_1 & 0 & -\frac12e & \frac14x_2 \\
x_2 & -\frac12e &0  & \frac14x_1 \\
2e & x_2& -x_1 & 0 \end{array} \right),\] 
which is flat.

Up to equivalence, there is the unique regular normal extension for $\mathfrak{so}(3)/\mathfrak{so}(2)$ to $\mathfrak{sp}(4,\mathbb{R})$ with
\[\alpha(e,x_1,x_2)=
\left( \begin{array}{cccc}
0 & -\frac14x_1& -\frac14x_2 & -\frac18e \\
x_1 & 0 & \frac12e & -\frac14x_2 \\
x_2 & -\frac12e &0  & \frac14x_1 \\
2e & x_2& -x_1 & 0 \end{array} \right),\] 
which is flat.
\end{theorem}

\section{Extensions to Lagrangean contact structures}

In this section we construct examples of Lagrangean contact structures with involutive smooth system of symmetries. We want to find extension to Cartan geometry of type $(\mathfrak{sl}(n+2,\mathbb{R}),P)$ with the following gradation, where the blocks are $(1,n,1)$:

\[ \left( \begin{array}{ccc}
\mathfrak{l}_{0} & \mathfrak{l}_{1} & \mathfrak{l}_{2} \\
\mathfrak{l}_{-1} & \mathfrak{l}_{0} & \mathfrak{l}_{1} \\
\mathfrak{l}_{-2} & \mathfrak{l}_{-1} & \mathfrak{l}_{0} \end{array} \right)\]

The representation of the semisimple part of $\mathfrak{l}_{0}$ on $\mathfrak{l}_{-1}$ is $V\oplus V^*$, where $V$ is standard representation of $\mathfrak{sl}(n,\mathbb{R})$ and $V^*$ is its dual.

Firstly we look on Lagrangean contact structures for simple symmetric spaces.

\begin{prop}\label{odw1}
The only non-complex simple symmetric spaces allowing extensions to Lagrangean contact structures are simple para-pseudo-hermitian symmetric space and the pseudo-hermitian symmetric spaces $\mathfrak{so}(p+2,q)/ \mathfrak{so}(p,q)+\mathfrak{so}(2)$. For the latter cases, the infinitesimal inclusion $i'$ from proposition \ref{pp29} is unique up to equivalence, and if there is no normal subgroup of $L$ in $P$, then the $i$ is unique up to equivalence.
\end{prop}
\begin{proof}
Let $G/H$ be a non-complex simple symmetric space and let $K$ be the simisimple part of $H$ extended by the symmetry $h$. Since in the para-pseudo-hermitian case, the $\mathfrak{k}$ has representation $W\oplus W^*$ for some irreducible representation $W: \mathfrak{k}\to \mathfrak{sl}(n,\mathbb{R})$, we define $i$ by $W$. Then $K$ and $i(K)$ are isomorphic, because $(V\oplus V^*)\circ W=V\circ W\oplus V^*\circ W=W\oplus W^*$. For the pseudo-hermitian symmetric spaces the same is possible only in the case of type $\mathbb{R}$ and $W^*\cong \bar{W}$.

Since semisimple part of $L_0$ is simple, we can use proposition \ref{pp30} and we see that $i$ is $W$ or $W^*$, up to equivalence. Then we define morphism $G\times_WP\to G\times_{W^*}P$ as $(g,p)\mapsto ((g^{-1})^T,p)$, which maps extension $(W,\alpha)$ to $(W^*,-\alpha^T)$, and the claim follows from proposition \ref{pp29}.
\end{proof}

Now we explicitly compute one flat example.

\begin{example}
Extension from $(PGl(n+1,\mathbb{R}),Gl(n,\mathbb{R}))$ to $(PGl(n+2,\mathbb{R}),P)$:

The subgroup $Gl(n,\mathbb{R})$ is represented by the following matrices, where the blocks are $(1,n)$ and $B\in Gl(n,\mathbb{R})$

\[ \left( \begin{array}{cc}
1 & 0  \\
0 & B  \end{array} \right).\] 

The symmetry at $o$ is a left multiplication by the following matrix in $Gl(n,\mathbb{R})$, where $E$ is the identity matrix

\[ \left( \begin{array}{cc}
1 & 0  \\
0 & -E  \end{array} \right).\] 

$K$ is the following subgroup, where $A\in Sl(n,\mathbb{R})$

\[ \left( \begin{array}{cc}
1 & 0  \\
0 & \pm A  \end{array} \right).\] 

Now $i$ is the following injective homomorphism, which maps $K$ into $P$

\[ \left( \begin{array}{ccc}
1 & 0 & 0 \\
0 & \pm A & 0 \\
0 & 0 & 1 \end{array} \right).\] 

Since both adjoint representations are $\lambda_1 \oplus \lambda_{n-1}$, the only possible homomorphisms are nonzero multiples. Thus the only possible $\alpha$ are the following, where $a=-Tr(A)$ and $b_1,b_2\in \mathbb{R}$ are nonzero and $c_1,c_2,d_1,d_2,e_1\in \mathbb{R}$

\[ \left( \begin{array}{cc}
a & Y^T \\ 
X & A   \end{array} \right) 
\mapsto 
\left( \begin{array}{ccc}
c_1a & d_1Y^T & e_1a \\ 
b_1X & A+\frac{c2}{n}Ea & d_2X \\
b_1b_2a & b_2Y^T & (1-c_1-c_2)a   \end{array} \right). \]

For fixed $b_1,b_2$ the normality conditions are equivalent to $c_2=0$, $e_1=d_1d_2$, $(n+2)b_1d_1+nd_2b_2-2c_1=n$ and $nb_1d_1-2c_1+(n+2)b_2d_2=n+2$. Thus there are four conditions on five variables and the solution is $d_1=\frac{c_1}{b_1}, d_2=-\frac{c_1-1}{b_2}, c_2=0,e_1=-\frac{c_1-1}{b_2}\frac{c_1}{b_1}$ and $c_1$ free parameter. Thus we can choose $c_1=\frac12$ and then the $\alpha$ extending to normal geometry for fixed $b_1,b_2$ is

\[ \left( \begin{array}{cc}
a & Y^T \\ 
X & A   \end{array} \right) 
\mapsto 
\left( \begin{array}{ccc}
\frac{1}{2}a & \frac{1}{2b_1}Y^T & \frac{1}{4b_1b_2}a \\ 
b_1X & A & \frac{1}{2b_2}X \\
b_1b_2a & b_2Y^T & \frac{1}{2}a   \end{array}\right).  \]

Further $\kappa_{\alpha}(X,Y)=0$ for any of these $\alpha$. So they are all equivalent and locally isomorphic to homogeneous model. We can summarize the results in the following proposition.

\begin{prop}
Up to equivalence, there is the unique regular normal extension from $(PGl(n+1,\mathbb{R}),K)$ to Lagrangean contact geometry, which is flat.
\end{prop}
\end{example}

In the case $W$ and $W^*$ are not isomorphic as the representations of $\mathfrak{k}$, then by the Schur lemma only the multiples of identity are isomorphisms. After identification of the representations of $\mathfrak{k}$ and $i(K)$ via $W$, we are in situation of the previous example. Since the symmetric space has now different curvature $R(X,Y)$, and $\kappa_{\alpha}(X,Y)=[\alpha(X),\alpha(Y)]-\alpha(R(X,Y))$, the resulting contact geometry will not be flat. But using morphism M1) we get that again they are all isomorphic. Thus we get the following theorem.

\begin{theorem}
Up to equivalence, there is an unique regular normal extension for any non-complex simple para-pseudo-hermitian symmetric space with $W\neq W^*$ to Lagrangean contact structure. The extended geometry is flat only in the case of the previous example.
\end{theorem}

We investigate two remaining cases with simple group generated by symmetries, where the representation $W$ is self dual.

\begin{example}
Extension from $(O(p+2,q),O(p,q)\times O(2))$ and $O(p+1,q+1),O(p,q)\times O(1,1))$ to $(PGl(n+2,\mathbb{R}),P)$:

The subgroups $O(p,q)\times O(2)$ and $O(p,q)\times O(1,1)$ are represented by the following matrices, where the blocks are $(2,n)$ and $B\in O(p,q)$ and $b\in O(2)$ or $b\in O(1,1)$

\[ \left( \begin{array}{cc}
b & 0  \\
0 & B  \end{array} \right).\] 

The symmetry at $o$ is represented by a left multiplication by the following matrix in $O(p,q)\times O(2)$ or $O(p,q)\times O(1,1)$, where $E$ are the identity matrices

\[ \left( \begin{array}{cc}
E & 0  \\
0 & -E  \end{array} \right).\] 

$K$ is the following subgroup, where $A\in O(p,q)$

\[ \left( \begin{array}{cc}
E & 0  \\
0 & A  \end{array} \right).\] 

Now $i$ is the following injective homomorphism, which maps $K$ into $P$

\[ \left( \begin{array}{ccc}
1 & 0 & 0 \\
0 & A & 0 \\
0 & 0 & 1 \end{array} \right).\] 

The adjoint representation of $K$ is $\lambda_1 \oplus \lambda_1$ and $i(K)$ is $\lambda_1 \oplus \lambda_{n-1}$, since $K=O(p,q)$, the $\lambda_{n-1}\cong \lambda_1$ as representation of $K$. Now the possible isomorphisms are maps $(X,Y)\mapsto (b_1X+b_2Y,b_3X+b_4Y)$ for $b_1b_4-b_2b_3\neq 0$. Thus the only possible $\alpha$ are the following, where $a\in \mathbb{R}$ and $c$ is $1$ in the $O(2)$ case and $-1$ in the $O(1,1)$ case, $I$ is matrix with $p$ entries on diagonal $1$ and remaining $q$ entries $-1$ and $c_1,c_2,d_1,d_2,d_3,d_4,e_1\in \mathbb{R}$

\[ \left( \begin{array}{ccc}
0 & a & -X^TI \\
-c a & 0 & -c Y^TI \\ 
X & Y & A   \end{array} \right) 
\\
\mapsto 
\\
\left( \begin{array}{ccc}
c_1a & d_1X^TI+d_2Y^TI & e_1a \\ 
b_1X+b_2Y & A+\frac{c_2a}{n}E & d_3X+d_4Y \\
(b_1b_4-b_2b_3)a & b_3X^TI+b_4Y^TI & (-c_1-c_2)a   \end{array} \right). \]

We denote $\gamma=cb_1b_3+b_2b_4$ and $\delta=b_1b_4-b_2b_3$. For fixed $b_1, b_2, \gamma, \delta$ the normality conditions are $c_2=\frac{n}{n+1}\frac{-\gamma}{\delta}$,$e_1=d_2d_3-d_1d_4$, $b_4d_1-b_3d_2=\frac{-b_4^2-cb_3^2}{\delta}$, $b_2d_3-b_1d_4=\frac{b_2^2+cb_1^2}{\delta}$, $b_2d_1-b_1d_2-b_4d_3+b_3d_4=\frac{n}{n+1}\frac{-\gamma}{\delta}$ and $b_2d_1-b_1d_2+b_4d_3-b_3d_4+2c_1=\frac{n}{n+1}\frac{\gamma}{\delta}$. Thus there are six conditions on seven variables and we compute the solution for $d_1, d_2, d_3, d_4, e_1$ and $c_2$ and let $c_1$ as a free parameter. Thus we can choose $c_1=\frac{n\gamma}{2(n+1)\delta}$ and the $\alpha$ extending to normal geometry for fixed $b_1, b_2, \gamma, \delta$ is
\[ \left( \begin{array}{ccc}
0 & a & -X^TI \\
-c a & 0 & -c Y^TI \\ 
X & Y & A   \end{array} \right) 
\mapsto 
\]
\[
\left( \begin{array}{ccc}
\frac{n\gamma}{2(n+1)\delta} a& V_1 &-(\frac{((3n+2)(n+2)\gamma^2}{4(n+1)^2\delta^3}+\frac{c}{\delta})a \\ 
b_1X+b_2Y & A-\frac{1}{n+1}\frac{\gamma}{\delta}Ea & V_2\\
\delta a & b_3X^TI+b_4Y^TI & \frac{n\gamma}{2(n+1)\delta} a   \end{array}\right),  \]
where 
\[ V_1=-(\frac{(n+2)\gamma b_3}{2(n+1)\delta^2}+\frac{b_4}{\delta})X^TI-(\frac{c(n+2)\gamma b_4}{2(n+1)\delta^2}-\frac{cb_3}{\delta})Y^TI, \] 
\[ V_2=-(\frac{(n+2)\gamma b_1}{2(n+1)\delta^2}-\frac{b_2}{\delta})X-(\frac{c(n+2)\gamma b_2}{2(n+1)\delta^2}+\frac{cb_1}{\delta})Y .\]

The curvature of the extended geometry by this $\alpha$ is:
\[ \kappa_{\alpha}\left( \left( \begin{array}{ccc}
0 & a & -X^TI \\
-c a & 0 & -c Y^TI \\ 
X & Y & 0   \end{array} \right) ,
\left( \begin{array}{ccc}
0 & b & -Z^TI \\
-c b & 0 & -c W^TI \\ 
Z & W & 0   \end{array} \right)  \right)=
\]
\[
\left( \begin{array}{ccc}
0 & -\frac{n+2}{n+1}\frac{(cb_3^2+b_4^2)\gamma}{\delta^3}V_3 & 0\\
0 & \gamma \frac{R_1\delta-\frac{(c+1)(n+2)}{2(n+1)}R_2+\frac{n+2}{2(n+1)}R_3}{\delta^2} -\frac{\gamma (W^TIX-Y^TIZ)}{(n+1)\delta} E & -\frac{n+2}{n+1}\frac{(cb_1^2+b_2^2)\gamma}{\delta^3}V_4 \\ 
0 & 0 & 0   \end{array} \right), \] 
where 
\[ R_1=XW^TI+WX^TI-YZ^TI-ZY^TI,\] 
\[ R_2=b_1b_4(XW^TI-YZ^TI)-b_2b_3(WX^TI-ZY^TI),\]
\[ R_3=b_1b_3(ZX^TI-XZ^TI)+b_2b_4(WY^TI-YW^TI\]
are $n\times n$ matrices and 
\[ V_3=b_1bX^TI+b_2bY^TI-b_1aZ^TI-b_2aW^TI,\] 
\[ V_4=b_3bX+b_4bY-b_3aZ-b_4aW\]
are matrices $1\times n$ and $n\times 1$.

For $\gamma=cb_1b_3+b_2b_4=0$ the extended geometry is flat. Using algorithm in proposition \ref{lab_1} we compute, that the infinitesimal automorphisms for $\gamma \neq 0$ are of the form $\alpha(\mathfrak{g})$, with exception of the case $c=-1,\ b_1^2=b_2^2,\ b_3^2=b_4^2$, when the infinitesimal automorphisms consists $\alpha(\mathfrak{g})$ and elements of center of $\mathfrak{l}_0$ with trivial action on $\mathfrak{l}_2$. The equivalence classes are determined by $(b_1,b_2,b_3,b_4)$ in the same way as in dimension $3$. In particular, $t=\frac{\gamma}{\delta}$.

Let us briefly discuss a geometric realization of the extension. Notice, that $G/K$ is a generalization of Stiefel variety, i.e. pairs of orthonormal vectors in $\mathbb{R}^{n+2}$ with induced metric of $O(2)$ and $O(1,1)$ in the plain given by this two vectors. The coefficients $(b_1,b_2)$ from definition of $\alpha$, together with the frame $(X,Y)$ in the Stiefel variety are interpreted as a vector in $\mathbb{R}^{n+2}$. Clearly, there is a $n$-dimensional subbundle of such points in the Stiefel variety leading to the same vector for the chosen coordinates $(b_1,b_2)$. Choosing other non-collinear coordinates $(b_3,b_4)$, we get for any point in the Stiefel variety two subbundles and the tangent bundles to them gives the Lagrangean contact structure. Let $\phi$ be angle between vectors of coordinates $(b_1,b_2)$ and $(b_3,b_4)$, which does not depend on the choice of basis. Then for $c=1$ (the $O(2)$ case), we get $t=\frac{\gamma}{\delta}=cotan(\phi)$ and for $c=-1$ (the $O(1,1)$ case), we get $t=\frac{\gamma}{\delta}=cotanh(\phi)$.

\begin{theorem}\label{5.6}
Up to equivalence, all regular normal extensions from $(O(p+2,q),O(p,q))$ and $O(p+1,q+1),O(p,q))$ in the case $b_1^2>b_2^2$ to a Lagrangean contact geometry are given by the following one parameter classes for $t\geq 0$:
\[\alpha \left( \begin{array}{ccc}
0 & a & -X^TI \\
-c a & 0 & -c Y^TI \\ 
X & Y & A   \end{array} \right) 
= 
\]
\[
\left( \begin{array}{ccc}
\frac{n}{2(n+1)}t a& -(\frac{(n+2)}{2(n+1)}ct^2+1)X^TI-\frac{n}{2(n+1)}tY^TI&-\frac{((3n+2)(n+2)}{4(n+1)^2}t^2a-ca \\ 
X & A-\frac{1}{n+1}t Ea &-\frac{(n+2)}{2(n+1)}tX-cY\\
a & c t X^TI+Y^TI & \frac{n}{2(n+1)} t a   \end{array} \right)  \]
with curvature
\[ \kappa_{\alpha}\left( \left( \begin{array}{ccc}
0 & a & -X^TI \\
-c a & 0 & -c Y^TI \\ 
X & Y & 0   \end{array} \right) ,
\left( \begin{array}{ccc}
0 & b & -Z^TI \\
-c b & 0 & -c W^TI \\ 
Z & W & 0   \end{array} \right)  \right)=
\]
\[
\left( \begin{array}{ccc}
0 & \frac{(n+2)t}{n+1}(1+ct^2)(bX^TI-aZ^TI) & 0\\
0 & \frac{t}{(n+1)}((n+1)R_1-R_2+(n+2)ctR_3) -\frac{tr_1}{(n+1)}E& -\frac{(n+2)t}{n+1}V_1 \\ 
0 & 0 & 0   \end{array} \right), \]
where 
\[ R_1=WX^TI-ZY^TI,\]
\[ R_2=XW^TI-YZ^TI,\]
\[ R_3=ZX^TI-XZ^TI,\]
\[ r_1=W^TIX-Y^TIZ,\]
\[ V_1=t(bX-aZ)+c(bY-aW).\]

For $b_1^2=b_2^2$ and $O(p+1,q+1),O(p,q))$:

a) for $b_3^2<b_4^2$
\[\alpha \left( \begin{array}{ccc}
0 & a & -X^TI \\
a & 0 & Y^TI \\ 
X & Y & A   \end{array} \right) 
= 
\]
\[
\left( \begin{array}{ccc}
\frac{n}{2(n+1)} a& -\frac{(n+2)}{2(n+1)}Y^TI-X^TI&-\frac{(3n+2)(n+2)}{4(n+1)^2}a+a \\ 
X+Y & A-\frac{1}{n+1} Ea &\frac{n}{2(n+1)}(X+Y)\\
a & Y^TI & \frac{n}{2(n+1)} a   \end{array} \right)  \]
with curvature
\[ \kappa_{\alpha}\left( \left( \begin{array}{ccc}
0 & a & -X^TI \\
a & 0 & Y^TI \\ 
X & Y & 0   \end{array} \right) ,
\left( \begin{array}{ccc}
0 & b & -Z^TI \\
b & 0 & W^TI \\ 
Z & W & 0   \end{array} \right)  \right)=
\]
\[
\left( \begin{array}{ccc}
0 & \frac{(n+2)}{n+1}(bY^TI-aZ^TI)+\frac{(n+2)}{n+1}(bX^TI-aW^TI) & 0\\
0 & \frac{1}{(n+1)}((n+1)R_1-R_2-(n+2)R_3) -\frac{r_1}{(n+1)}E& 0 \\ 
0 & 0 & 0   \end{array} \right), \]
where 
\[ R_1=WX^TI-YZ^TI,\]
\[ R_2=XW^TI-ZY^TI,\]
\[ R_3=YW^TI-WY^TI,\]
\[ r_1=W^TIX-Y^TIZ.\]

b) for $b_3^2=b_4^2$
\[\alpha \left( \begin{array}{ccc}
0 & a & -X^TI \\
a & 0 & Y^TI \\ 
X & Y & A   \end{array} \right) 
= 
\]
\[
\left( \begin{array}{ccc}
\frac{n}{2(n+1)} a& -\frac{n}{4(n+1)}X^TI+\frac{n}{4(n+1)}Y^TI&\frac{n^2}{8(n+1)^2}a \\ 
X+Y & A-\frac{1}{n+1} Ea &\frac{n}{4(n+1)}(X+Y)\\
2a & -X^TI+Y^TI & \frac{n}{2(n+1)} a   \end{array} \right)  \]
with curvature
\[ \kappa_{\alpha}\left( \left( \begin{array}{ccc}
0 & a & -X^TI \\
a & 0 & Y^TI \\ 
X & Y & 0   \end{array} \right) ,
\left( \begin{array}{ccc}
0 & b & -Z^TI \\
b & 0 & W^TI \\ 
Z & W & 0   \end{array} \right)  \right)=
\]
\[
\left( \begin{array}{ccc}
0 & 0 & 0\\
0 & \frac{1}{2(n+1)}(nR_1+nR_2+(n+2)R_3-(n+2)R_4) -\frac{r_1}{(n+1)}E& 0 \\ 
0 & 0 & 0   \end{array} \right), \]
where 
\[ R_1=XW^TI-ZY^TI,\]
\[ R_2=WX^TI-YZ^TI,\]
\[ R_3=XZ^TI-ZX^TI,\]
\[ R_4=YW^TI-WY^TI,\]
\[ r_1=W^TIX-Y^TIZ.\]
\end{theorem}
\end{example}

The classification in the semisimple case is the following:
 
\begin{theorem}
The only semisimple non-simple symmetric spaces without complex factors allowing extensions to Lagrangean contact structures are semisimple para-pseudo-hermitian symmetric spaces. For the latter cases, the infinitesimal inclusion $i'$ from proposition \ref{pp29} is unique up to equivalence, and if there is no normal subgroup of $L$ in $P$, then the $i$ is unique up to equivalence.
\end{theorem}
\begin{proof}
For semisimple para-pseudo-hermitian symmetric spaces without complex factors, the extension can be done in two steps. First we take extension from the sum of symmetric spaces to the structure group $(Gl(n,\mathbb{R})\times Gl(n,\mathbb{R}))\cap O(n,n)$, which acts as standard and dual to standard representation and is unique up to para-complex multiple. Then the claim follows in the same way as proposition \ref{odw1}.

Now assume the extension exists. Then since the representation of $i(K)$ is completely reducible, the simple factors have extension to Lagrangean contact geometry, when we restrict to the submatrix (in basis compatible with factors) with values in this factor. This defines extension to Lagrangean contact geometry of lower dimension. Assume that one factor is pseudo-hermitian and not para-hermitian, then the eigenvalues of its center are $\pm i$ and $H/K$ has to be this center, which is contradiction since due to regularity the $H/K$ intersects all factors.
\end{proof}

\bf Geometrical interpretation. \rm

As described in \cite{odk14} one can relate Lagrangean contact geometry with system of differential equations. In our case the relation is as follows.

Let $G/K$ be a reflexion space with underlying semisimple symmetric space $G/H$ and $(\alpha,i)$ extension to Lagrangean contact geometry $(p: G\times_i P\to G/K, \omega_\alpha)$. Let $E$ be
\[ Tp \circ \omega_\alpha^{-1} \left( \begin{array}{ccc}
\mathfrak{l}_{0} & \mathfrak{l}_{1} & \mathfrak{l}_{2} \\
\mathfrak{l}_{-1} & \mathfrak{l}_{0} & \mathfrak{l}_{1} \\
0 & 0 & \mathfrak{l}_{0} \end{array} \right)\]
and let $V$ be
\[ Tp \circ \omega_\alpha^{-1} \left( \begin{array}{ccc}
\mathfrak{l}_{0} & \mathfrak{l}_{1} & \mathfrak{l}_{2} \\
0 & \mathfrak{l}_{0} & \mathfrak{l}_{1} \\
0 & \mathfrak{l}_{-1} & \mathfrak{l}_{0} \end{array} \right).\]

Then, since the latter Cartan geometry is torsion-free, the distributions $E,V$ are integrable. Since $i(K)\subset L_0$, these distributions are invariant with respect to $K$ action i.e. they are given by $\mathfrak{e},\mathfrak{v}\subset \mathfrak{g}/\mathfrak{k}$ and the leaf space corresponding to $V$ is homogeneous space $M=G/exp(\mathfrak{v})$. Now the Cartan geometry corresponds to system of differential equations on $M$. The space of solutions is then homogeneous space $S=G/exp(\mathfrak{e})$ and the correspondence is as follows: For the point of $S$ there is $g\cdot exp(\mathfrak{e})$ orbit in $G$, which projects to hyperspace in $M$. Thus the symmetry group of differential equation is $G$ (if the geometry is not flat).

\begin{example}
Extension from $(O(p+2,q),O(p,q)\times O(2))$ to $(PGl(n+2,\mathbb{R}),P)$. If $\alpha$ is given as in theorem \ref{5.6}, then $\mathfrak{e}$ is given by $a=0, t X^TI+Y^TI=0$ and $\mathfrak{v}$ is given by $a=0, X=0$. Thus both $M$ and $S$ are $O(p+2,q)/O(p+1,q)$ i.e. quadric in $R^{n+2}$. The correspondence is as follows: The point $g\cdot exp(\mathfrak{e})$ is associated with the intersection of quadric with hyperplane through $g\cdot O(p+1,q)$ orthogonal (in the metric defining quadric) to $g\cdot (t,1,0,\dots,0)$.
\end{example}

\begin{example}
Extension from $(O(p+1,q+1),O(p,q)\times O(1,1))$ to $(PGl(n+2,\mathbb{R}),P)$. There are three possible non-equivalent $\alpha$.

a) $b_1^2>b_2^2$

Then $\mathfrak{e}$ is given by $a=0, -t X^TI+Y^TI=0$ and $\mathfrak{v}$ is given by $a=0, X=0$. Now $M$ is $O(p+1,q+1)/O(p+1,q)$, if $t>1$ then $S$ is $O(p+1,q+1)/O(p+1,q)$, if $t<1$ then $S$ is $O(p+1,q+1)/O(p,q+1)$ and if $t=1$ then $S$ is $O(p+1,q+1)/(O(p,q)\ltimes R^n)$ i.e. again quadric in $R^{n+2}$. The correspondence is as follows: The point $g\cdot exp(\mathfrak{e})$ is associated with the intersection of quadric with hyperplane through $g$ orthogonal (in the metric defining quadric) to $g\cdot (-t,1,0,\dots,0)$.

b) $b_1^2=b_2^2$ and $b_3^2<b_4^2$

Then $\mathfrak{e}$ is given by $a=0, Y=0$ and $\mathfrak{v}$ is given by $a=0, Y=-X$. Now $M$ is $O(p+1,q+1)/(O(p,q)\ltimes R^n)$ and $S$ is $O(p+1,q+1)/O(p+1,q)$. The correspondence is as follows: The point $g\cdot exp(\mathfrak{e})$ is associated with the intersection of quadric with hyperplane through $g$ orthogonal (in the metric defining quadric) to $g\cdot (0,1,0,\dots,0)$.

c) $b_1^2=b_2^2$ and $b_3^2=b_4^2$

Then $\mathfrak{e}$ is given by $a=0, X=Y$ and $\mathfrak{v}$ is given by $a=0, Y=-X$. Now $M$ is $O(p+1,q+1)/(O(p,q)\ltimes R^n)$ and $S$ is $O(p+1,q+1)/(O(p,q)\ltimes R^n)$. The correspondence is as follows: The point $g\cdot exp(\mathfrak{e})$ is associated with the intersection of quadric with hyperplane through $g$ orthogonal (in the metric defining quadric) to $g\cdot (-1,1,0,\dots,0)$.
\end{example}

\section{Extensions to CR structures}

In this section we construct examples of partially integrable almost CR structures with smooth system of symmetries, i.e. due to the torsion freeness we construct the CR structures. So we want to find extension to Cartan geometry of type $(\mathfrak{su}(p+1,q+1),P)$ with the following gradation:
\[ \left( \begin{array}{ccc}
\mathfrak{l}_{0} & \mathfrak{l}_{1} & \mathfrak{l}_{2} \\
\mathfrak{l}_{-1} & \mathfrak{l}_{0} & \mathfrak{l}_{1} \\
\mathfrak{l}_{-2} & \mathfrak{l}_{-1} & \mathfrak{l}_{0} \end{array} \right),\]
where the blocks are $(1,n,1)$ and $AJ+JA^*=0$ for $A\in \mathfrak{su}(p+1,q+1)$, where $J$ is representing the pseudo hermitian form $$(x_0,x_i,x_{n+1})J(y_0,y_i,y_{n+1})^*=x_0\bar{y}_{n+1}+x_{n+1}\bar{y}_{0}+\sum_{i=1}^p x_i\bar{y}_{i}-\sum_{i=p+1}^n x_i\bar{y}_{i}.$$

The representation of the semisimple part of $\mathfrak{l}_{0}$ on $\mathfrak{l}_{-1}$ is $V$, where $V$ is standard representation of $\mathfrak{su}(p,q)$.

\begin{prop}\label{odw2}
The only non-complex simple symmetric spaces allowing extensions to CR structures are simple pseudo-hermitian symmetric spaces and simple para-pseudo-hermitian symmetric spaces $\mathfrak{so}(p+1,q+1)/ \mathfrak{so}(p,q)+\mathfrak{so}(1,1)$. For the latter cases, the infinitesimal inclusion $i'$ from proposition \ref{pp29} is unique up to equivalence, and if there is no normal subgroup of $L$ in $P$, then the $i$ is unique up to equivalence.
\end{prop}
\begin{proof}
Let $G/H$ be a non-complex simple symmetric space and let $K$ be the simisimple part of $H$ extended by the symmetry $h$. In the pseudo-hermitian case, the $\mathfrak{k}$ has representation $W$ on $\mathfrak{m}$ for some representation $W: \mathfrak{k}\to \mathfrak{su}(p,q)$ and we can define $i$ by $W$. Then $K$ and $i(K)$ are isomorphic, because $V\circ W=W$. In the para-pseudo-hermitian case, the same is possible only if $W^*\cong \bar{W}$.

Since semisimple part of $L_0$ is simple, we can use proposition \ref{pp30} and we see that $i$ is up to equivalence $W$ or $\bar{W}$. Then we define morphism $G\times_WP\to G\times_{\bar{W}}P$ as $(g,p)\mapsto ((g^{-1})^*,p)$, which maps extension $(W,\alpha)$ on $(\bar{W},-\alpha^*)$, and the claim follows from proposition \ref{pp29}.
\end{proof}

Now we explicitly compute one flat example.

\begin{example}
Extension from $(PSU(p+1,q),U(p,q))$ to $(PSU(p+1,q+1),P)$:

The subgroup $U(p,q)$ is represented by the following matrixes, where the blocks are $(1,n)$ and $B\in U(p,q)$

\[ \left( \begin{array}{cc}
1 & 0  \\
0 & B  \end{array} \right)\] 

The symmetry at $o$ is a left multiplication by the following matrix in $PSU(p+1,q)$, where $E$ is the identity matrix
\[ \left( \begin{array}{cc}
1 & 0  \\
0 & -E  \end{array} \right).\] 

$K$ is the following subgroup, where $A\in SU(p,q)$
\[ \left( \begin{array}{cc}
1 & 0  \\
0 & \pm A  \end{array} \right).\] 

Now $i$ is the following injective homomorphism, which maps $K$ into $P$
\[ \left( \begin{array}{ccc}
1 & 0 & 0 \\
0 & \pm A & 0 \\
0 & 0 & 1 \end{array} \right).\] 

Since both adjoint representations are $\lambda_1$, the only possible homomorphisms are nonzero (complex) multiples. Thus the only possible $\alpha$ are the following, where $a=-Tr(A)$ and $b\in \mathbb{C}$ is nonzero, $c,d \in \mathbb{C}$ and $e\in \mathbb{R}$:
\[ \left( \begin{array}{cc}
a & -\bar{X}^TI \\ 
X & A   \end{array} \right) 
\mapsto 
\left( \begin{array}{ccc}
ca & -\bar{d}\bar{X}^TI & ea \\ 
bX & A+\frac{1-2Re(c)}{n}Ea & dX \\
b\bar{b}a & -\bar{b}\bar{X}^TI & \bar{c}a   \end{array} \right). \]

For fixed $b$, the normality conditions are equivalent to $Re(c)=1/2$, $e=d\bar{d}$, $c=\bar{b}d$. Thus there are four conditions on five variables and the solution is $d=\frac{c}{\bar{b}}, Re(c)=0,e=\frac{c\bar{c}}{b\bar{b}}$ and $Im(c)$ free parameter. Thus if we choose $Im(c)=0$, the resulting $\alpha$ for fixed $b$ is:
\[ \left( \begin{array}{cc}
a & -\bar{X}^TI \\ 
X & A   \end{array} \right) 
\mapsto 
\left( \begin{array}{ccc}
\frac12 a & -\frac{1}{2b}\bar{X}^TI & \frac{1}{4b\bar{b}}a \\ 
bX & A & \frac{1}{2\bar{b}}X \\
b\bar{b}a & -\bar{b}\bar{X}^TI & \frac12 a   \end{array}\right).  \]

Further $\kappa_{\alpha}(X,Y)=0$ for any of these $\alpha$. So they are all isomorphic and locally isomorphic to homogeneous model. We can summarize the result in the following proposition.

\begin{prop}
Up to equivalence, there is an unique regular normal extension from $(PSU(p+1,q),K)$ to CR structure, which is flat.
\end{prop}
\end{example}

In the case $W$ and $\bar{W}$ are not isomorphic as the representations of $K$, then by the Schur lemma only the multiples of identity are isomorphisms. After identification of the representations of $\mathfrak{k}$ and $i(K)$ via $W$, we are in situation of the previous example. Since the symmetric space has now different curvature $R(X,Y)$, and $\kappa_{\alpha}(X,Y)=[\alpha(X),\alpha(Y)]-\alpha(R(X,Y))$, the resulting contact geometry will not be flat. But using morphism M1) we get that again they are all isomorphic. Thus we get the following theorem.

\begin{theorem}
Up to equivalence, there is an unique regular normal extension for any non-complex simple pseudo-hermitian symmetric space with $W\neq\bar{W}$ to CR structure. The extended geometry is flat only in the case of the previous example.
\end{theorem}

Now we investigate the remaining cases with simple group generated by symmetries, where $W$ is self conjugate.

\begin{example}
Extension from $(O(p+2,q),O(p,q)\times O(2))$, $(O(p+1,q+1),O(p,q)\times O(1,1))$ to $(PSU(p+1,q+1),P)$:

The symmetric space and the $i$ are the same as in Lagrangean contact case.

Since there is no complex structure on $K$, we choose two identifications of $W=\lambda_1+\lambda_1$ with complex numbers, i.e. $(X_1,X_2)\mapsto X_1+iX_2=X$ and $(X_1,X_2)\mapsto X_2+iX_1=-i\bar{X}$. Then the isomorphisms of representations is given by complex multiples of those two identifications by $b_1,b_2\neq 0$ such, that $|b_1|\neq |b_2|$. So all the possible $\alpha$ are the following, where $a\in \mathbb{R}$ and $c$ is $1$ in the $O(2)$ case and $-1$ in the $O(1,1)$ case, $I$ is matrix with $p$ entries on diagonal $1$ and remaining $q$ entries $-1$, $c_1,d_1,d_2\in \mathbb{C}$ and $e_1\in \mathbb{R}$:
\[ \left( \begin{array}{ccc}
0 & a & -X_1^TI \\
-c a & 0 & -c X_2^TI \\ 
X_1 & X_2 & A   \end{array} \right) 
\mapsto \]
\[
\left( \begin{array}{ccc}
c_1a & -(\bar{d}_1\bar{X}^T+\bar{d}_2iX^T)I & e_1ai \\ 
b_1X-b_2i\bar{X} & A-\frac{2Im(c_1)}{n}Eai & d_1X-d_2i\bar{X} \\
2(|b_1|^2-|b_2|^2)ai & -(\bar{b}_1\bar{X}^T+\bar{b}_2iX^T)I & -\bar{c_1}a   \end{array} \right). \]

For fixed $b_1,b_2$, the normality conditions are different for $c=1$ and $c=-1$, so we skip the exact form of them. We only mention, that $Re(c_1)$ is a free parameter and we choose $Re(c_1)=0$. The resulting $\alpha$ for fixed $b_1,b_2$ is:

For $c=-1$
\[ \left( \begin{array}{ccc}
0 & a & -X_1^TI \\
a & 0 & X_2^TI \\ 
X_1 & X_2 & A   \end{array} \right) \mapsto
\]
\[ 
\left( \begin{array}{ccc}
\frac{n}{2(n+1)}tai & * & (\frac{-1}{2(|b_1|^2-|b_2|^2)}-\frac{(n+2)(3n+2)}{8(n+1)^2}\frac{t^2}{|b_1|^2-|b_2|^2})ai \\ 
b_1X-b_2i\bar{X} & A-\frac{2}{2(n+1)}tEai & V_1 \\
2(|b_1|^2-|b_2|^2)ai & * & \frac{n}{2(n+1)}tai   \end{array} \right), \]
where $t:=\frac{(b_1,b_2)}{|b_1|^2-|b_2|^2}=2\frac{Re(b_1)Im(b_2)-Re(b_2)Im(b_1)}{|b_1|^2-|b_2|^2}$, entry on $*$ comes from structure of Lie algebra $\mathfrak{su}(p+1,q+1)$ and 
\[V_1=(\frac{ib_2}{2(|b_1|^2-|b_2|^2)}-\frac{n+2}{4(n+1)}\frac{tb_1}{|b_1|^2-|b_2|^2})X\]\[+(\frac{ib_1}{2(|b_1|^2-|b_2|^2)}+\frac{n+2}{4(n+1)}\frac{tb_2}{|b_1|^2-|b_2|^2})i\bar{X}.\]

For $c=1$
\[ \left( \begin{array}{ccc}
0 & a & -X_1^TI \\
-a & 0 & -X_2^TI \\ 
X_1 & X_2 & A   \end{array} \right) 
\mapsto 
\]
\[
\left( \begin{array}{ccc}
\frac{n}{2(n+1)}tai & * & (\frac{1}{2(|b_1|^2-|b_2|^2)}-\frac{(n+2)(3n+2)}{8(n+1)^2}\frac{t^2}{|b_1|^2-|b_2|^2})ai \\ 
b_1X-b_2i\bar{X} & A-\frac{1}{n+1}tEai & V_2 \\
2(|b_1|^2-|b_2|^2)ai & * & \frac{n}{2(n+1)}tai   \end{array} \right), \]
where $t:=\frac{(b_1,b_2)}{|b_1|^2-|b_2|^2}=\frac{|b_1|^2+|b_2|^2}{|b_1|^2-|b_2|^2}$, entry on $*$ comes from structure of Lie algebra $\mathfrak{su}(p+1,q+1)$ and 
\[V_2=(\frac{b_1}{2(|b_1|^2-|b_2|^2)}-\frac{n+2}{4(n+1)}\frac{tb_1}{|b_1|^2-|b_2|^2})X\]\[+(\frac{b_2}{2(|b_1|^2-|b_2|^2)}+\frac{n+2}{4(n+1)}\frac{tb_2}{|b_1|^2-|b_2|^2})i\bar{X}. \]

Explicit computation of the curvature using Maple reveals, that $\kappa=0$ for $t=0$, and $\kappa\neq 0$ otherwise. Using the algorithm from proposition \ref{lab_1} we compute, that the infinitesimal automorphisms for $t\neq 0$ are of the form $\alpha(\mathfrak{g})$, with exception of the case $c=1,\ t=1$, when the infinitesimal automorphisms consists $\alpha(\mathfrak{g})$ and elements of the form $$\left( \begin{array}{ccc}
li &0 & 0 \\ 
0 & -\frac{2l}{n}Ei & 0 \\
0 & 0 & li  \end{array} \right)$$ for $l\in \mathbb{R}$. Further using morphisms M1), M2) and M3) we get that $\alpha$ can be chosen for $c=-1$ with $b_1=\sqrt{\frac{1+\sqrt{t^2+1}}{2}}, b_2=i\sqrt{\frac{-1+\sqrt{t^2+1}}{2}},t>-1$ and for $c=1$ with $b_1=\sqrt{\frac{1+t}{2}}, b_2=\sqrt{\frac{t-1}{2}}i, t\geq 1$. Thus we can summarize:

\begin{theorem}
Up to equivalence, all regular normal extensions from $(O(p+2,q),O(p,q))$ to CR structures form one parameter class for $t\geq 1$.

Up to equivalence, all regular normal extensions from $O(p+1,q+1),O(p,q))$ to CR structures form one parameter class for $t>-1$.
\end{theorem}

We remark that in dimension three all homogeneous CR-geometries were found by Cartan in \cite{odk15}. As generalization of the defining functions found by Cartan, we conjecture that in $(O(p+2,q),O(p,q))$ case, the CR-hypersurface is given by equation $$1+\sum_{i=1}^p|z_i|^2-\sum_{i=p+1}^n|z_i|^2+|w|^2=t|1+\sum_{i=1}^pz_i^2-\sum_{i=p+1}^nz_i^2+w^2|$$ in $\mathbb{C}^{n+1}$, and in $O(p+1,q+1),O(p,q))$ case, the CR-hypersurface is given by equation $$1+\sum_{i=1}^p|z_i|^2-\sum_{i=p+1}^n|z_i|^2-|w|^2=t|1+\sum_{i=1}^pz_i^2-\sum_{i=p+1}^nz_i^2-w^2|$$ in $\mathbb{C}^{n+1}$.
\end{example}

\begin{example}
Extension from $(SO^*(2n+2),SO^*(2n)\times SO^*(2))$ to $(PSU(n,n),P)$:

We will not give the explicit form of $i$, the symmetric spaces and explicit computations, which were done using Maple,, but we start already with the $\alpha$. The representation $\lambda_1$ of $SO^*(2n)$ is quaternionic and the isomorphism are of the form 
\[ (f_1,f_2): X=X_1+iX_2+jX_3+kX_4\mapsto (X_1+iX_2,X_3+iX_4)\]
up to right quaternionic multiple. We also skip details on the computation of normality conditions and present the $\alpha$ leading the regular normal extension: 
\[ \left( \begin{array}{cccc}
0 & -X_1^T-iX_2^T & ai & -X_3^T+iX_4^T \\
X_1+iX_2 & A+iB & X_3+iX_4& C+iD \\ 
ai & X_3^T+iX_4^T & 0 & -X_1^T-iX_2^T \\
-X_3+iX_4 & -C+iD & X_1-iX_2& A-iB \\   \end{array} \right) 
\mapsto\]
\[
\left( \begin{array}{cccc}
\frac{nt}{(2n+1)|b|}ai & -f_1(\bar{X}\bar{d})^T & f_2(\bar{X}\bar{d})^T &|d|ai \\
f_1(Xb) & A-Di-\frac{1t}{(2n+1)|b|}aiE & B-Ci& f_1(Xd) \\ 
f_2(Xb) & -B-Ci & A+iD-\frac{1t}{(2n+1)|b|}aiE & f_2(Xd) \\
|b|ai & -f_1(\bar{X}\bar{b})^T & f_2(\bar{X}\bar{b})^T& \frac{nt}{(2n+1)|b|}ai \\   \end{array} \right), 
\]
where $b=b_1+ib_2+jb_3+kb_4\neq 0$, $t=b_1^2-b_2^2-b_3^2+b_4^2$ and 
\[d=\frac{(b_1+kb_4)((2n+1)|b|-(n+1)t)}{(2n+1)|b|^2}+\frac{(ib_2+jb_3)((n+1)|b|-(2n+1)t)}{(2n+1)|b|^2}.\]

The extension is flat for $t=0$ and non-flat otherwise. Using algorithm from proposition \ref{lab_1} we compute, that the infinitesimal automorphisms for $t\neq 0$ are of the form $\alpha(\mathfrak{g})$. Further using morphism M1) and M2) we get that the $\alpha$ can be chosen with $b=\sqrt{\frac{1+t}{2}}+\sqrt{\frac{1-t}{2}}j$.

\begin{theorem}
Up to equivalence, all regular normal extensions from $(SO^*(2n+2),SO^*(2n))$ to CR structures form one parameter class for $t\geq 0$. They are non flat for $t\neq 0$.
\end{theorem}
\end{example}

The classification in the semisimple case is following.

\begin{theorem}
The only semisimple non-simple symmetric spaces without complex factors allowing extensions to CR structures are semisimple pseudo-hermitian symmetric spaces. For the latter cases, the infinitesimal inclusion $i'$ from proposition \ref{pp29} is unique up to equivalence, and if there is no normal subgroup of $L$ in $P$, then the $i$ is unique up to equivalence.
\end{theorem}
\begin{proof}
For semisimple pseudo-hermitian symmetric space without complex factors, the extension can be done in two steps. First we take extension from the sum of symmetric spaces to the structure group $U(p,q)$, which acts as standard representation. Then the claim follows in the same way as proposition \ref{odw2}.

Now assume the extension exists. Then for the same reasons as in the Lagrangian case, the simple factors have extension to integrable almost CR structures. Assume that one factor is para-pseudo-hermitian and not pseudo-hermitian, then the eigenvalues of its center are $\pm 1$ and $H/K$ has to be this center, which is contradiction since due to regularity the $H/K$ intersects all factors.
\end{proof}

\section{Extension to contact projective structures}

In this section we construct examples of contact projective structures with a smooth system of symmetries. I.e. we find extensions to Cartan geometry of type $(\mathfrak{sp}(2n+2,\mathbb{R}),P)$ with the following gradation:
\[ \left( \begin{array}{ccc}
\mathfrak{l}_{0} & \mathfrak{l}_{1} & \mathfrak{l}_{2} \\
\mathfrak{l}_{-1} & \mathfrak{l}_{0} & \mathfrak{l}_{1} \\
\mathfrak{l}_{-2} & \mathfrak{l}_{-1} & \mathfrak{l}_{0} \end{array} \right),\]
where the blocks are $(1,2n,1)$ and $AJ+JA^T=0$ for $A\in \mathfrak{sp}(2n+2,\mathbb{R})$, where $J$ is representing the symplectic form $$(x_0,x_i,x_{2n+1})J(y_0,y_i,y_{2n+1})^*=x_0y_{2n+1}+x_{2n+1}y_{0}+\sum_{i=1}^n (x_iy_{n+i}- x_{n+i}y_{i}).$$

The representation of the semisimple part of $\mathfrak{l}_{0}$ on $\mathfrak{l}_{-1}$ is the standard representation of $\mathfrak{sp}(2n,\mathbb{R})$.

\begin{prop}
The only non-complex simple symmetric spaces allowing extensions to contact projective structures are simple para-pseudo-hermitian or pseudo-hermitian symmetric spaces. For these cases, the infinitesimal inclusion $i'$ from proposition \ref{pp29} is unique up to equivalence, and if there is no normal subgroup of $L$ in $P$, then the $i$ is unique up to equivalence.
\end{prop}
\begin{proof}
Let $G/H$ be a non-complex simple symmetric space and $K$ simisimple part of $H$ extended by $h$. For simple pseudo-hermitian symmetric spaces, the $i'$ is 
\[
\left( \begin{array}{cc}
0& 0 \\ 
0 & A+iB   \end{array} \right) 
\mapsto
 \left( \begin{array}{cccc}
0 & 0 & 0 & 0 \\
0 & A & -BI& 0 \\ 
0 & IB & IAI & 0 \\
0 & 0 & 0& 0 \\   \end{array} \right), 
\]
where $I$ is diagonal matrix given by the signature of the metric as before, $IA+A^TI=0$ and $IB-B^TI=0$. 

For simple para-pseudo-hermitian symmetric spaces, the $i'$ is inclusion of $\mathfrak{so}(n,n)$ as a subgroup.

The element $i(h)$ is 
\[ \left( \begin{array}{cccc}
-1 & 0 & 0 & 0 \\
0 & E & 0& 0 \\ 
0 & 0 & E & 0 \\
0 & 0 & 0& -1 \\   \end{array} \right).
\]

Then the representations of $K$ and $i(K)$ are isomorphic and the extension exists from proposition \ref{pp29}. Since semisimple part of $L_0$ is simple, we can use proposition \ref{pp30} and we see that $i$ is unique up to equivalence.
\end{proof}

In the same way as for the previous types of geometries, we conclude the following theorem. We consider representation $W$ as in Theorems \ref{odw1} or \ref{odw2}.

\begin{theorem}
If the representation $W$ is not self dual in the para-pseudo-hermitian case or not self-conjugate in the pseudo-hermitian case, then there is (up to equivalence) unique regular normal contact projective structure for this non-complex simple (para)-pseudo-hermitian symmetric space.
\end{theorem}

Now we compute the simple examples, where $W$ is self-dual or self-conjugate.

\begin{example}
Extension from $(O(p+2,q),O(p,q)\times O(2))$, $(O(p+1,q+1),O(p,q)\times O(1,1))$ to $(PSp(2n+2,\mathbb{R}),P)$:

The symmetric spaces are the same as in the case of the previous structures. All possible $\alpha$ are: 
\[ \left( \begin{array}{ccc}
0 & a & -X_1^TI \\
-ca & 0 & -cX_2^TI \\ 
X_1 & X_2 & A   \end{array} \right) 
\mapsto\]
\[
\left( \begin{array}{cccc}
c_1a & * & * & e_1a \\
b_1X_1+b_2X_2 & A+c_2aE & gaI& d_3X_1+d_4X_2 \\ 
b_3X_1I+b_3X_2I & haI & IAI-c_2aE & -d_1X_1I-d_2X_2I \\
2(b_1b_4-b_2b_3)a & * & *& -c_1a \\   \end{array} \right),
\]
where entries on $*$ comes from structure of the Lie algebra $\mathfrak{sp}(2n+2,\mathbb{R})$ and all coefficients are real numbers such, that $b_1b_4-b_2b_3\neq 0$.

For fixed $b$'s, the normality conditions give us, that $c_1$ can be chosen as free parameter and remaining parameters are dependent. Using the morphisms M1) and M2), we get, that all choices of $b$'s are isomorphic. So we get the following result:

\begin{theorem}
Up to equivalence, there is the unique regular normal extension from $(O(p+2,q),O(p,q))$ or $(O(p+1,q+1),O(p,q))$ to contact projective structures given by:
\[ \left( \begin{array}{ccc}
0 & a & -X_1^TI \\
-ca & 0 & -cX_2^TI \\ 
X_1 & X_2 & A   \end{array} \right) 
\mapsto
\left( \begin{array}{cccc}
0 & \frac{-n}{2(n+1)}(X_1I)^T & \frac{-cn}{2(n+1)}X_2^T & \frac{-2cn^2}{4(n+1)^2}a \\
X_1 & A & \frac{1}{n+1}aI& \frac{-cn}{2(n+1)}X_2 \\ 
X_2I & \frac{-1}{n+1}aI & IAI & \frac{n}{2(n+1)}X_1I \\
2a & (X_2I)^T & -X_1^T& 0 \\   \end{array} \right)
\]
with curvature 
\[ \kappa(\left( \begin{array}{ccc}
0 & a & -X_1^TI \\
-ca & 0 & -cX_2^TI \\ 
X_1 & X_2 & A   \end{array} \right),
 \left( \begin{array}{ccc}
0 & b & -Y_1^TI \\
-cb & 0 & -cY_2^TI \\ 
Y_1 & Y_2 & B   \end{array} \right) )
=\]
\[
\left( \begin{array}{cccc}
0 & 0 & 0 & 0 \\
0 & R_1 & -cR_3I-R_2I& 0 \\ 
0 &  IR_3+IR_2 & IR_1I & 0 \\
0 & 0 & 0& 0 \\   \end{array} \right),
\]
where \[ R_1=\frac{n+2}{2(n+1)}(X_1Y_1^T-Y_1X_1^T+cX_2Y_2^T-cY_2X_2^T),  \]
\[R_2=\frac{1}{(n+1)}(X_2^TY_1-X_1^TY_2),\]
\[R_3=\frac{n}{2(n+1)}(X_1Y_2^T+Y_2X_1^T-X_2Y_1^T-Y_1X_2^T). \]
\end{theorem}
\end{example}

\begin{example}
Extension from $(SO^*(2n+2),SO^*(2n)\times SO^*(2))$ to $(PSp(2n+2,\mathbb{R}),P)$:

Technical computations using Maple lead to the following theorem. We skip the exact form of the symmetric spaces. The representation $\lambda_1$ of $SO^*(2n)$ is quaternionic and the isomorphism is 
\[ X_1+iX_2+jX_3+kX_4\mapsto (X_1,X_2,X_3,X_4) \]
up to a quaternionic multiple. We also skip the details on computation of normality conditions and computation of automorphisms and isomorphisms here. 

\begin{theorem}
Up to equivalence, there is unique regular normal extensions from $(SO^*(2n+2),SO^*(2n))$ to contact projective structures given by:
\[ \left( \begin{array}{cccc}
0 & -X_1^T-iX_2^T & ai & -X_3^T+iX_4^T \\
X_1+iX_2 & A+iB & X_3+iX_4& C+iD \\ 
ai & X_3^T+iX_4^T & 0 & -X_1^T-iX_2^T \\
-X_3+iX_4 & -C+iD & X_1-iX_2& A-iB \\   \end{array} \right) 
\mapsto\]
\[
\left( \begin{array}{cccccc}
0 & *& * & *& * & -\frac{4n^2}{(2n+1)^2}a \\
X_1 & A& -B& -D_1& -C& -\frac{2n}{(2n+1)}X_3 \\ 
X_2 & B& A& -C& D_1& -\frac{2n}{(2n+1)}X_4 \\
X_3 & D_1& C& A& -B&\frac{2n}{(2n+1)}X_1\\ 
-X_4 & C& -D_1& B& A & \frac{-2n}{(2n+1)}X_2 \\
a & * & *& * & *& 0 \\   \end{array} \right),
\]
where $D_1=D-\frac{a}{(2n+1)}E$ and entries on $*$ comes from structure of the Lie algebra $\mathfrak{sp}(2n+2,\mathbb{R})$, with curvature
\[ \kappa(\left( \begin{array}{cccc}
0 & -X_1^T-iX_2^T & ai & -X_3^T+iX_4^T \\
X_1+iX_2 & 0 & X_3+iX_4& 0 \\ 
ai & X_3^T+iX_4^T & 0 & -X_1^T-iX_2^T \\
-X_3+iX_4 & 0 & X_1-iX_2& 0 \\   \end{array} \right),\]
\[
\left( \begin{array}{cccc}
0 & -Y_1^T-iY_2^T & bi & -Y_3^T+iY_4^T \\
Y_1+iY_2 & 0 & Y_3+iY_4& 0 \\ 
bi & Y_3^T+iY_4^T & 0 & -Y_1^T-iY_2^T \\
-Y_3+iY_4 & 0 & Y_1-iY_2& 0 \\   \end{array} \right) )
=\]
\[
\left( \begin{array}{cccccc}
0 & 0 & 0& 0 & 0 & 0 \\
0 & R1 & R3& R5 & R7 & 0 \\
0 & R4 & R2& R8 & R6 & 0 \\
0 & -R5 & R7& R1 & -R3 & 0 \\
0 & R8 & -R6 & -R4 & R2 & 0 \\
0 & 0 & 0& 0 & 0 & 0 \\  \end{array} \right),
\]
where 
\[ R1=\frac{1}{(2n+1)}(X_1Y_1^T-Y_1X_1^T+X_4Y_4^T-Y_4X_4^T)-(X_2Y_2^T-Y_2X_2^T+X_3Y_3^T-Y_3X_3^T), \] 
\[ R2=(X_1Y_1^T-Y_1X_1^T+X_4Y_4^T-Y_4X_4^T)-\frac{1}{(2n+1)}(X_2Y_2^T-Y_2X_2^T+X_3Y_3^T-Y_3X_3^T), \]
\[ R3=-\frac{1}{(2n+1)}(X_1Y_3^T-Y_1X_3^T+X_4Y_2^T-Y_4X_2^T)-(X_2Y_4^T-Y_2X_4^T+X_3Y_1^T-Y_3X_1^T), \]
\[ R4=(X_1Y_3^T-Y_1X_3^T+X_4Y_2^T-Y_4X_2^T)+\frac{1}{(2n+1)}(X_2Y_4^T-Y_2X_4^T+X_3Y_1^T-Y_3X_1^T), \]
\[ R5=\frac{1}{(2n+1)}(X_1Y_4^T-Y_1X_4^T-X_4Y_1^T+Y_4X_1^T)-(X_2Y_3^T-Y_2X_3^T-X_3Y_2^T+Y_3X_2^T)\]\[-\frac{2}{(2n+1)}(X_1^TY_4-X_4^TY_1+X_2^TY_3-X_3^TY_2)E, \]
\[ R6=-(X_1Y_4^T-Y_1X_4^T-X_4Y_1^T+Y_4X_1^T)+\frac{1}{(2n+1)}(X_2Y_3^T-Y_2X_3^T-X_3Y_2^T+Y_3X_2^T)\]\[-\frac{2}{(2n+1)}(X_1^TY_4-X_4^TY_1+X_2^TY_3-X_3^TY_2)E, \]
\[ R7=\frac{1}{(2n+1)}(X_1Y_2^T-Y_1X_2^T-X_4Y_3^T+Y_4X_3^T)+(X_2Y_1^T-Y_2X_1^T-X_3Y_4^T+Y_3X_4^T), \]
\[ R8=-(X_1Y_2^T-Y_1X_2^T-X_4Y_3^T+Y_4X_3^T)-\frac{1}{(2n+1)}(X_2Y_1^T-Y_2X_1^T-X_3Y_4^T+Y_3X_4^T).\]
\end{theorem}
\end{example}

The classification in the semisimple case is the following:

\begin{theorem}
The only semisimple symmetric spaces without complex factors allowing extensions to contact projective structures are sums of simple (para)-pseudo-hermitian symmetric spaces. For the latter cases, the infinitesimal inclusion $i'$ from proposition \ref{pp29} is unique up to equivalence, and if there is no normal subgroup of $L$ in $P$, then the $i$ is unique up to equivalence.
\end{theorem}
\begin{proof}
Apart the center of $\mathfrak{h}$ the extension can be taken as in previous examples. If we have in mind, that any multiplication on invariant subspaces of $\mathfrak{l}_{-1}$ can be obtained by bracket with an element of  $\mathfrak{l}_{0}$, which commutes with image of semisimple part of $\mathfrak{h}$, then image of center of $\mathfrak{h}$ can be chosen to be such elements with appropriate action. The $\mathfrak{h}/\mathfrak{k}$ is then a sum of preimages of $\mathfrak{l}_{-2}$ parts of the relevant previous examples.
\end{proof}

\end{document}